\documentclass[a4,10pt]{article}
\usepackage{amsmath,amsthm,amssymb}
\usepackage[dvipdfmx]{graphicx}
\usepackage{fancybox}
\usepackage{color}
\usepackage{tikz}
\usepackage{subfig}
\usepackage{soul}
\usepackage{enumerate}
\usepackage{booktabs}

\newtheorem{Th}{Theorem}[section] 
\newtheorem{Lem}[Th]{Lemma} 
\newtheorem{Prop}[Th]{Proposition} 
 
\newtheorem{Def}[Th]{Definition} 
\newtheorem{Prob}[Th]{Problem} 
\newtheorem{Rem}[Th]{Remark}

\def\e{\varepsilon}
\def\N{{\mathbb N}}

\def\R{{\mathbb R}}

\newcommand{\supp}{\operatorname{supp}}

\makeatletter
\def\argmin{\mathop{\operator@font argmin}}
\makeatother

\begin{document}
\title{Particle dynamics with elastic collision at the boundary: existence and partial uniqueness of solutions}

\author{M.~Kimura \and P.~van Meurs \and Z.~Yang}

\date{}

\maketitle

%\tableofcontents

\begin{abstract}
We consider the dynamics of point particles which are confined to a bounded, possibly nonconvex domain $\Omega$. Collisions with the boundary are described as purely elastic collisions. This turns the description of the particle dynamics into a coupled system of second order ODEs with discontinuous right-hand side. The main contribution of this paper is to develop a precise solution concept for this particle system, and to prove existence of solutions. In this proof we construct a solution by passing to the limit in an auxiliary problem based on the Yosida approximation. In addition to existence of solutions, we establish a partial uniqueness theorem, and show by means of a counterexample that uniqueness of solutions cannot hold in general.
\end{abstract} 

\section{Introduction}\label{Intro}

Our aim is to establish existence and uniqueness of solutions to a system of particles with inertia, without damping, confined to a domain $\Omega \subset \R^m$, and subject to a particle interaction force (see Figure~\ref{domain}). 
\begin{figure}[htbp]
  \begin{center} 
    \includegraphics[width=80mm]{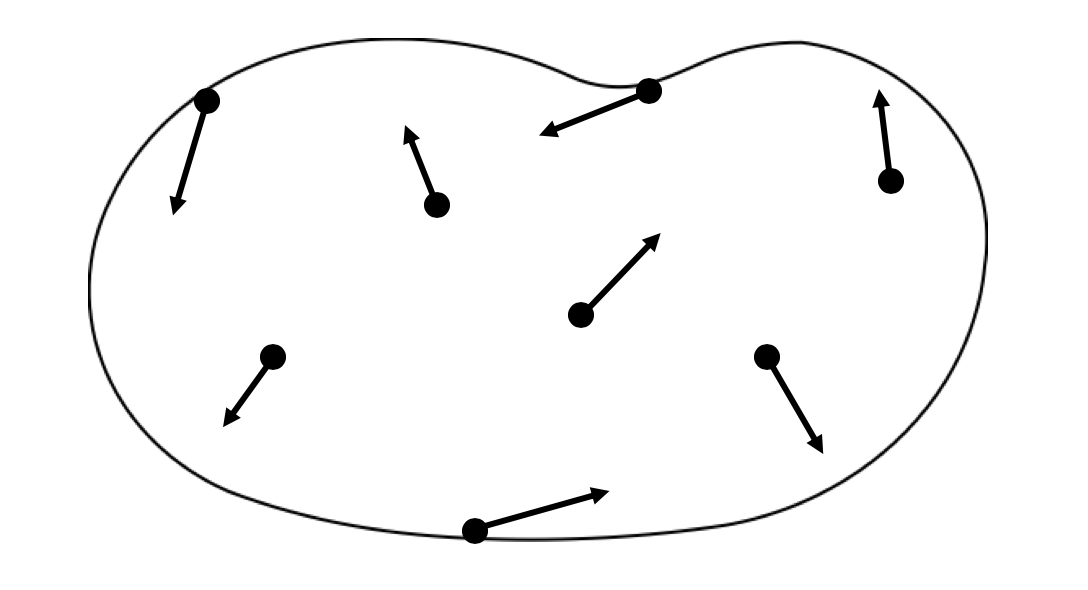}
  \end{center}
  \caption{Particles $x_i$ confined in domain $\Omega$ with force $F_i$.}
  \label{domain}
\end{figure}
The main interest is in the collision of the particles with the boundary $\partial \Omega$. We describe these collisions by elastic reflections, i.e., upon collision, the tangential velocity of the particle is conserved, but the normal velocity changes sign (see Figure \ref{case1}). The ramifications of this seemingly natural collision rule are delicate, as it allows for particles to move along the boundary (see Figure \ref{case2}. In this case, the collision rule gives rise to a normal force which keeps the particles confined), and uniqueness of solutions becomes tricky.
\begin{figure}[htbp]
\centering
\begin{minipage}[t]{0.48\textwidth}
\centering
\includegraphics[width=70mm]{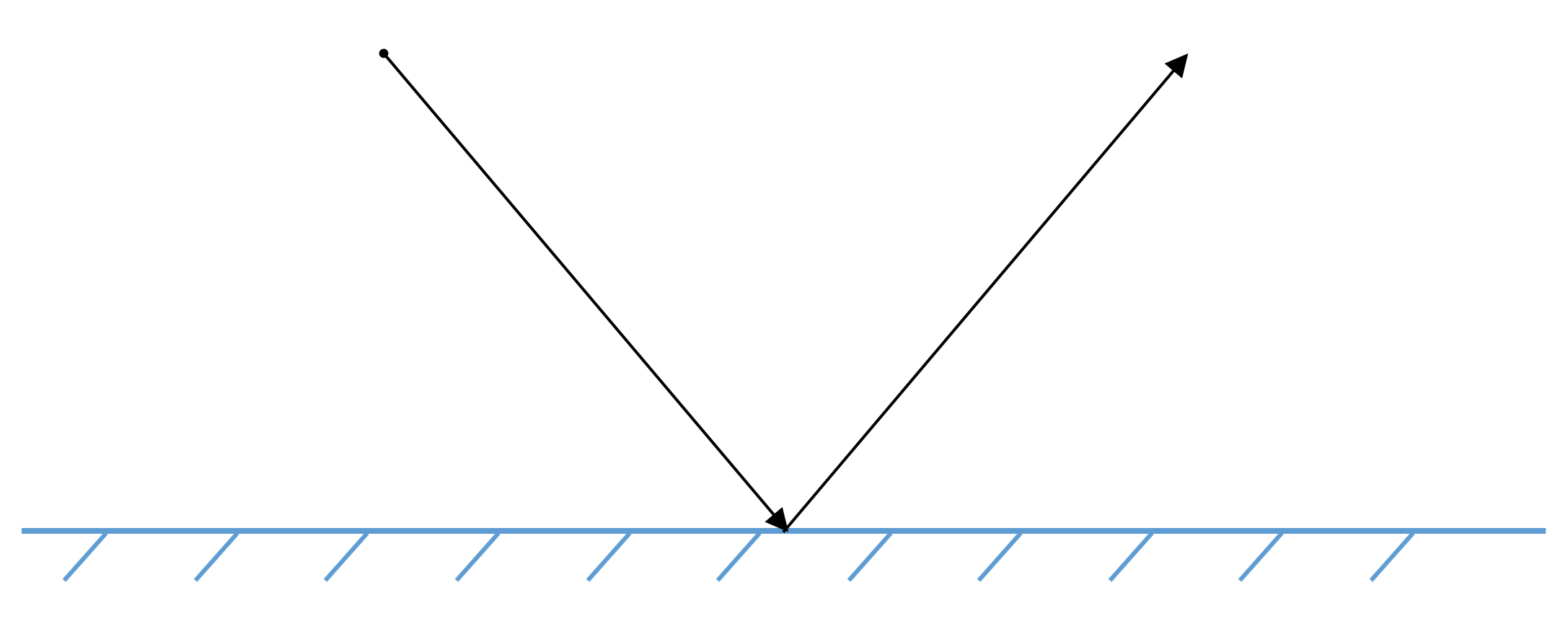}
\caption{}
\label{case1}
\end{minipage}
\begin{minipage}[t]{0.48\textwidth}
\centering
\includegraphics[width=70mm]{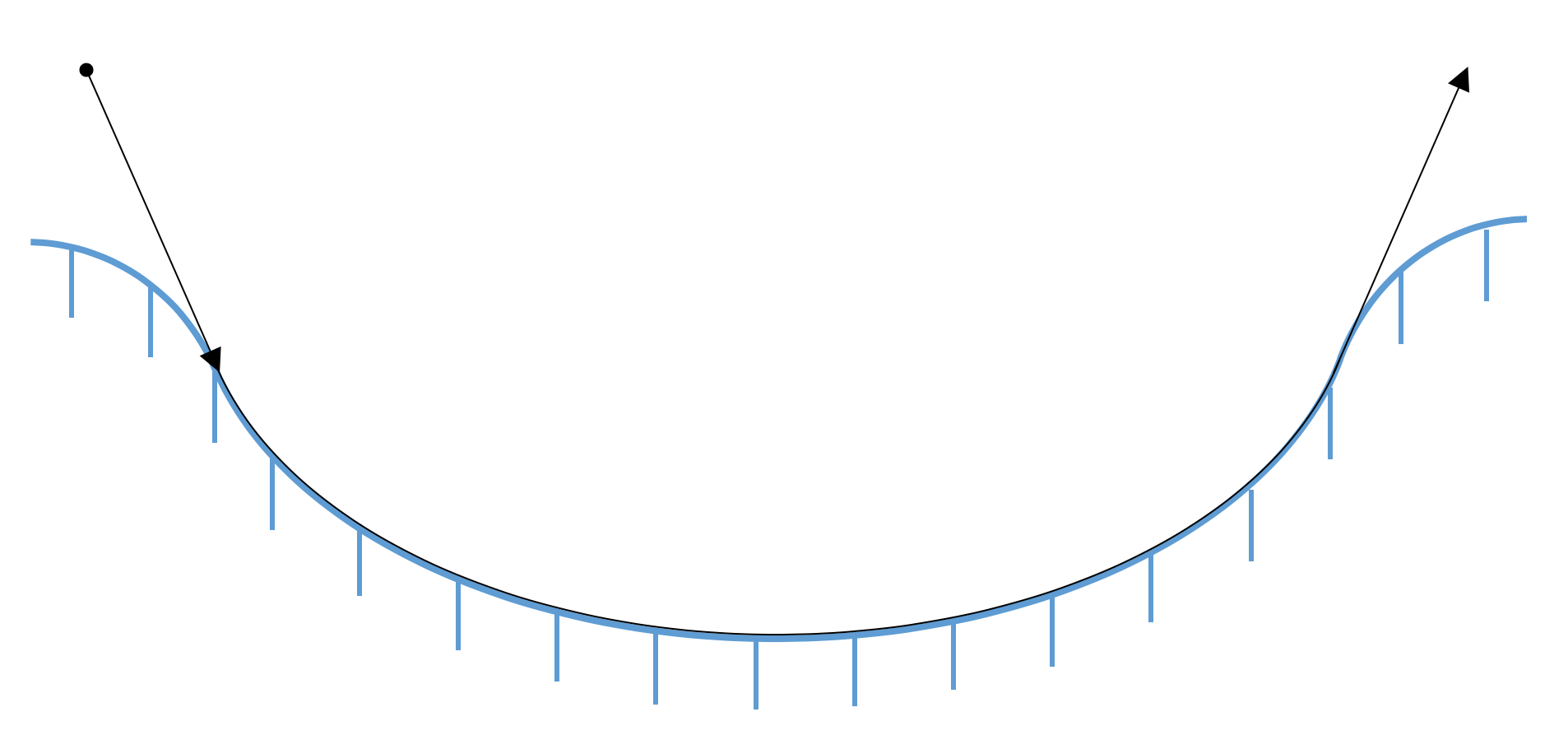}
\caption{}
\label{case2}
\end{minipage}
\end{figure}

The application that we have in mind for \eqref{P:informal} is the development of a properly motivated boundary condition in particle simulations for fluids. Indeed, in many such simulations (in particular when SPH (Smoothed Particle Hydrodynamics) \cite{gingold1977smoothed,li2007meshfree,liu2010smoothed,lucy1977numerical,monaghan1992smoothed,asai2012stabilized} is used), there are various kinds of ad hoc treatments of boundary conditions such as the use of ghost/dummy particles, wall particles, wall forces, etc. (see \cite{asai2012stabilized} and references therein). Yet, there is no established theory on how to treat properly boundary conditions for the particles which effectively model slip and no-slip conditions in the equations for the fluid. 

We describe the particle system in more detail, while leaving the precise description to Section~\ref{sec2}. We consider $n$ particles $x_1,\ldots ,x_n$ in $\overline\Omega$, and set $X:=\{x_i\}_{i=1}^n \in \overline\Omega^n \subset (\R^m)^n$ as the array of the particle's positions. For each particle $x_i$, we let the interaction force $F_i(t, X)$ depend on time and on all particle positions. The evolution of the particle system is then described by
\begin{equation} \label{P:informal}
  \ddot x_i(t) = F_i(t, X(t)) + \mu_i(t) \qquad (t > 0, \ i = 1,\ldots,n) 
\end{equation}
subject to boundary conditions at $t = 0$. The force $\mu_i(t)$ is treated as an unknown, and is to be determined from the reflection rule. Since the reflection rule only applies when $x_i(t) \in \partial \Omega$, we require $\mu_i(t) = 0$ for all $t$ such that $x_i(t) \in \Omega$.

Equation \eqref{P:informal} can be considered as a special case in which there is no damping. In previous work \cite{kimura2019particle}, we studied the overdamped limit, which is essentially obtained from \eqref{P:informal} by replacing $\ddot x_i$ by $\dot x_i$. Then, the natural confinement of the particles to the boundary does not result in a reflection rule, but in movement of particles along the boundary. From this point of view, we consider the present paper as the next step towards a general description of particle dynamics confined to bounded domains. 

Next we review relevant literature on the development of a solution concept to \eqref{P:informal}. As is obvious from the sketch in Figure \ref{case1}, we cannot expect a classical $C^2$-solution $X(t)$, and therefore we cannot rely on classical ODE theory such as the Cauchy-Peano Theorem and the Picard-Lindel\"of (Cauchy-Lipschitz) Theorem \cite{hartmanordinary}. In fact, the sketch in Figure \ref{case1} implies that $\mu_i(t)$ should have a Dirac-$\delta$ singularity at the time of collision, which means that we cannot expect $\mu_i$ to be a function, and therefore \eqref{P:informal} has to be interpreted in a weak sense.

To the best of our knowledge, the first solution concept to \eqref{P:informal} was developed by Schatzman \cite{schatzman1978class} in the framework of convex energies. In this setting, it is required that $\Omega$ is convex and that there exists an energy $E$ such that $F_i(t, X) = -\nabla_{x_i} E(t, X)$. Then, the description of the reflection rule can be transformed into an energy-conservation condition. The main result in \cite{schatzman1978class} is global existence of solutions, and a \emph{counterexample} to uniqueness of solutions. This counterexample is constructed for a scenario as in Figure \ref{case2}. We come back to this in Section~\ref{sec4}. After \cite{schatzman1978class}, Percivale \cite{percivale1991uniqueness} proved existence and uniqueness of solutions to \eqref{P:informal} for a restricted class of $F$ and for possibly nonconvex domains $\Omega$ which can be described as the $0$-level set of an analytic function. For well-posedness results of \eqref{P:informal} in a one-dimensional setting we refer to \cite{buttazzo1983approximation,carriero1980uniqueness,schatzman1998uniqueness}.

Our setting and goal is different from these results for two reasons. First, while \cite{schatzman1978class} requires a convex domain $\Omega$ with conservative field $F$, we allow for possibly nonconvex domains $\Omega$ and for nonconservative fields $F$. Since in this paper $\Omega$ is only required to be of class $C^3$ and $F$ is only required to be continuous, our setting is more general than that in \cite{percivale1991uniqueness} where analiticity of $F$ and $\Omega$ is required. Second, while we do not have a global-in-time uniqueness result as in \cite{percivale1991uniqueness}, we treat the issue of uniqueness more explicitly than in \cite{schatzman1978class} by specifying the time until which solutions are unique, and by constructing an explicit counterexample for global-in-time uniqueness.

In addition to these aims, we also develop a notion of energy conservation,  show that in particular cases the energy of solutions is conserved, and verify that both solutions in our counterexample conserve the energy. This shows that restricting the definition of solutions to those that conserve the energy cannot resolve the non-uniqueness issue.  

Our main mathematical contribution is the proof for the existence of solutions. This proof is based on \cite{schatzman1978class} and our earlier paper \cite{kimura2019particle} on the overdamped limit of \eqref{P:informal}. We construct a solution from an approximation of \eqref{P:informal} where the reflection is replaced by a regular force which pushes the particles back towards $\Omega$ whenever they leave $\Omega$. This approximate problem is based on the Yosida approximation of energies, and attains a classical solution. The core of the proof is to show that, in the limit where the size of the regular force outside of $\Omega$ tends to $\infty$, any limiting solution is a solution of \eqref{P:informal}. 

The paper is organised as follows. In Section \ref{sec2}, we define a precise solution concept for the particle system, and state our main theorems on the existence, partial uniqueness and other properties of solutions. In Section~\ref{sec3} we prove the existence of solutions. In Section~\ref{sec4}, we give a proof of partial uniqueness, and construct a counterexample to uniqueness of solutions in the general case. In Section \ref{sec5} we introduce a notion of energy conservation, and show to which extend it applies to our solution concept. Section \ref{sec:conc} contains the conclusion.

%%%%%%%%%%%%%%%%%%%%%%%%%%%%%%%%%%%%%%%%%%%%%%%%%%%%%%%%%%%%%%%%%%%%%%%%%%%%%%%%%%%%%%%%%%%%

\section{Definition and main results}\label{sec2}
\setcounter{equation}{0}
Let $\Omega$ be a bounded domain in $\R^m$ with $C^3$ boundary $\partial\Omega$. We also denote the outward unit normal vector on $\partial \Omega$ by $\nu$. We consider $n$ particles $x_1,\ldots ,x_n$ in $\overline\Omega$. We put $X:=\{x_i\}_{i=1}^n \in \overline\Omega^n \subset (\R^m)^n$ as the array of the particle's positions. On $(\R^m)^n$ we define the norm
\[
\|X\| := \max_{1\le i\le n}|x_i|,
\]
where $|x_i|$ is the Euclidean distance in $\R^m$.
We denote by $X(t) = \{x_i(t)\}_{i=1}^n$ the particle positions in time. 

Let us fix on a certain particle $x_i$ with $i \in \{1,\ldots, n\}$ for the moment. The acceleration $\ddot{x}_i$ of this particle is given by the force $F_i(t, X(t))$, where 
\begin{equation}\label{1}
F_i\in C([0,T]\times \overline{\Omega}^n;\R^m)\quad(i=1,\ldots,n)
\end{equation}
is given. Here and henceforth, for any function $f \in C(K; \R^m)$ for some compact $K \subset \R^\ell$ we define the norm
\[
  \| f \|_{C(K)} := \max_{y \in K} |f(y)|,
\]
which we abbreviate by $\| f \|_\infty$ if the domain of $f$ is clear from the context.
In some cases, we impose that $F_i$ is Lipschitz continuous in $X$, i.e., that there exists an $L > 0$ such that
\begin{equation}\label{2}
|F_i(t,X)-F_i(t,Y)|\le L\|X-Y\| \quad (t\in [0,T],~X,Y\in \overline{\Omega}^n,~i=1,\ldots,n).
\end{equation} 
If $x_i$ is at the boundary $\partial \Omega$, then we impose the following ``reflection rule":
\begin{equation}\label{reflection}
\frac{d^+ x_i}{dt}(t) = \frac{d^- x_i}{dt}(t) - 2\left( \frac{d^- x_i}{dt}(t)\cdot\nu(x_i(t))\right)\nu(x_i(t)).
\end{equation}
Here, 
\begin{equation} \label{lrder}
\frac{d^{\pm}}{dt}x_i(t):= \lim_{h\to 0+}\frac{x_i(t\pm h) - x_i(t)}{\pm h}
\end{equation}
denote the left- and right-derivatives.

\begin{Prob}\label{ProbP}
Let $n \in \N$ and $T > 0$. For each $i=1, \ldots, n$, let $F_i$ be as in \eqref{1}, $x_i^{\circ} \in \Omega$ and $v_i^{\circ} \in \R^m$. Find a solution $X(t)=\{x_i(t)\}_{i=1}^n$ to
\begin{equation}\label{P}
\left\{
\begin{aligned}
\ddot{x}_i(t) &= F_i(t, X(t)) + \text{``reflection rule \eqref{reflection} on $\partial \Omega$"}\quad(t\in (0,T)),\\
x_i(0) &= x_i^{\circ}, \\
\frac{dx_i}{dt}(0) &= v_i^{\circ}
\end{aligned}
\right.
\end{equation}
for $i = 1,\ldots,n$. 
\end{Prob}

Since the reflection rule results typically in discontinuous velocities, classical solutions to Problem \ref{ProbP} will not exist in general. Therefore, we introduce a weak notion of solutions in Definition \ref{Def}. To define it, we set $\mathcal{M}_+((0,T))$ as the space of finite, non-negative Borel measures on $(0,T)$.

\begin{Def}\label{Def}
We say that $X(t)=\{x_i(t)\}_{i=1}^n$ is a solution to Problem~\ref{ProbP} if for all $i\in\{1,\ldots,n\}$, $x_i$ satisfies
\begin{enumerate}[(i)]
\item
\[
\begin{aligned}
x_i & \in W^{1,\infty}(0,T; \R^m),\\
x_i(t) & \in\overline\Omega \qquad (t\in [0,T]),\\
\dot{x}_i & \in BV(0,T;\R^m), \\
x_i(0) &= x_i^\circ, \quad \frac{d^+x_i}{dt}(0) = v_i^{\circ}.
\end{aligned}
\]
\item
Let $I_{\partial\Omega}^i:=\{t\in [0,T]~;~x_i(t)\in\partial\Omega\}$. There exists $\rho_i\in \mathcal{M}_+((0,T))$ with $supp~\rho_i\subset I_{\partial\Omega}^i$ such that in the sense of distributions, 
\[
\ddot{x}_i(t) = F_i(t,X(t)) - \rho_i(t)\nu(x_i(t)),
\]
i.e., for any test function $\psi \in C_0^{\infty}(0,T; \R^m)$, 
\begin{equation}\label{weakform}
-\int_0^T\dot{x}_i(t)\cdot \dot{\psi}(t) dt 
= \int_0^T F_i(t,X(t))\cdot\psi(t)dt 
  - \int_{I_{\partial\Omega}^i} \nu(x_i(t)) \cdot \psi(t) d\rho_i(t).
\end{equation}

\item $x_i$ satisfies the reflection rule \eqref{reflection} for all $t\in I_{\partial\Omega}^i$.
\end{enumerate}
\end{Def}

\begin{Rem}\label{remarkop}
If $X$ is a solution to Problem \ref{ProbP}, then the weak form \eqref{weakform} holds for a larger class of test functions $\psi$. Indeed, since each of the three terms in the weak form viewed as an operator on $\psi$ is a bounded linear operator from $W^{1,1}(0,T;\R^m)$ to $\R$, and since $C_0^{\infty}(0,T;\R^m)$ is dense in $W_0^{1,1}(0,T;\R^m)$, \eqref{weakform} is satisfied for any $\psi\in W_0^{1,1}(0,T;\R^m)$ and any $i=1,\ldots,n$.
\end{Rem}

First, we explain the precise interpretation of Definition \ref{Def}(i). This is needed, for instance, in Definition \ref{Def}(iii) to make sure that the corresponding left- and right-derivatives defined in \eqref{lrder} exist. By writing $x_i \in W^{1,\infty}(0,T; \R^m)$ we mean that $\dot x_i$ is the weak derivative of $x_i$, and that $x_i$ is the Lipschitz continuous representative. Furthermore, for any $u\in BV(0,T;\R^m)$, we recall from  \cite{ambrosio2000functions} Theorem 3.28 that there exists a representative $u$ for which the left- and right-limits 
\begin{equation} \label{rllims:BV}
   u(t \pm) := \lim_{h \to 0+} u(t \pm h)
 \end{equation} 
are defined for all $t \in [0,T]$ (for $t = 0$ and $t = T$, only the right- and left-limit are defined, respectively). One possible choice for this representative is given by the right-continuous function
\[
u(t) := u_0 + Du([0,t]) \quad (t \in [0,T])
\]
for a certain constant vector $u_0 \in \R^m$.
In Definition \ref{Def}(i) and elsewhere in the paper, we consider this representative.

The following proposition ties together the weak derivative and the left- and right-derivatives defined in \eqref{lrder}.

\begin{Prop}\label{Proplim}
If $x\in W^{1,\infty}(0,T;\R^m)$ and $\dot{x}\in BV(0,T;\R^m)$, then the left- and right-derivatives defined in \eqref{lrder} exist, and 
\begin{align*}
  \frac{d^{+}}{dt}x(t) &= \dot{x}(t+) \quad (t\in [0,T)), \\
  \frac{d^{-}}{dt}x(t) &= \dot{x}(t-) \quad (t\in (0,T]). \\
\end{align*}
\end{Prop}

\begin{proof}
For $t\in [0,T)$, we obtain
\[
\begin{aligned}
\lim_{h\to 0+} \frac{x(t + h) - x(t)}{h}
 =\lim_{h\to 0+} \frac{1}{h} \int_t^{t + h} \dot{x} (s) ds
 =\dot{x}(t+),
\end{aligned}
\]
and thus $\frac{d^+ x}{dt} (t) = \dot{x}(t+)$.
The argument for the left-derivative is analogous.
\end{proof}

Our main result are the following theorems. They guarantee global existence and local uniqueness of the solutions to Problem~\ref{ProbP} defined in \ref{Def}.
\begin{Th}\label{exTh}
Let $T>0$, $X^{\circ}:=\{x_i^{\circ}\}_{i=1}^n \in \Omega^n$ and $V^{\circ}:=\{v_i^{\circ}\}_{i=1}^n \in (\R^m)^n$. If $F_i$ satisfies \eqref{1}, then there exists a solution $X$ as in Definition~\ref{Def} to Problem~\ref{ProbP}. 
\end{Th}

\begin{Th}\label{uniTh}
Given the setting of Theorem \ref{exTh}, let $X$ be a solution to Problem~\ref{ProbP}, and set 
\[
T_0:= \min_{1\le i\le n} \inf \{ t\in I_{\partial\Omega}^i~;~\dot{x}_i(t-)\cdot \nu(x_i(t))=0\}. 
\]
If $F_i$ satisfies \eqref{2}, then any solution $Y$ to Problem~\ref{ProbP} coincides with $X$ at least up to $T_0$, i.e., $Y|_{[0,T_0]} = X|_{[0,T_0]}$.
\end{Th}

We end this section with establishing two properties of solutions to Problem \ref{ProbP}. The first one is the observation that collision with the boundary happens always wihting $\overline \Omega$. The proof is obvious; we omit it. 

\begin{Prop}\label{lemofsol}
If $X$ is a solution to Problem~\ref{ProbP},  then for any $i\in \{1,\ldots,n\}$ and any $t\in I_{\partial\Omega}^i$, $x_i$ satisfies
\[
\frac{d^-}{dt}x_i(t)\cdot \nu(x_i(t) ) \ge 0\quad \text{and}\quad\frac{d^+}{dt}x_i(t)\cdot \nu(x_i(t) ) \le 0
\]
\end{Prop}

The following proposition states that, while the velocity of the particles may be discontinuous, the speed is continuous. 

\begin{Prop}\label{speed}
Let $X$ be a solution to Problem \ref{ProbP}. Then, for any $i\in \{1,\ldots,n\}$, there exist  $\alpha_i \in C^0([0,T])$ such that
\begin{equation}\label{lrlimit}
\alpha_i(t)=\left| \frac{d^+}{dt}x_i(t)\right| = \left| \frac{d^-}{dt}x_i(t)\right|\quad (t\in (0,T)).
\end{equation}
We denote $|\dot{x}_i|(t):=\alpha_i(t)$ and call it the speed of particle $x_i$ at $t$.
\end{Prop}

\begin{proof}
Fix $i\in\{1,\ldots,n\}$. If $t \in (0,T)$ is such that $x_i(t)\in \Omega$, then by continuity there exists a $\delta>0$ such that $x_i(s)\in \Omega$ for all $|s-t| < \delta$. Hence, $\supp \rho_i \cap (t_0-\delta,t_0+\delta) = \emptyset$, and thus $x_i \in C^2((t_0-\delta,t_0+\delta) ;\R^m)$. \eqref{lrlimit} follows. 

If $t \in (0,T)$ is such that $x_i(t)\in \partial\Omega$, then \eqref{reflection} holds. For any vector $\tau \in \R^m$ tangent to $\partial \Omega$ at $x_i(t)$, we obtain from \eqref{reflection} that
\[
  \frac{d^+}{dt}x_i(t) \cdot \tau =  \frac{d^-}{dt}x_i(t) \cdot \tau.
\]
For the normal direction we obtain
\[
  \frac{d^+}{dt}x_i(t) \cdot \nu(x_i(t)) =  -\frac{d^-}{dt}x_i(t)  \cdot \nu(x_i(t)).
\]
By Pythagoras' Theorem, \eqref{lrlimit} follows.

Next we claim that $\frac{d^+}{dt}x_i$ is right continuous on $[0,T)$ and that  $\frac{d^-}{dt}x_i$ is left continuous on $(0,T]$. Proposition \ref{speed} follows directly from this claim and \eqref{lrlimit}. We prove the claim first for $\frac{d^+}{dt}x_i$. Let $t \in [0,T)$ and $0 < h < T-t$. Recalling that $\dot x_i$ satisfies \eqref{rllims:BV}, we obtain from Proposition \ref{Proplim} that
\begin{align*}
  \left| \frac{d^+}{dt} x_i(t)\right|
  &= |\dot{x}_i(t+)|
  = \lim_{\e \to 0+} |\dot{x}_i(t+\e)|, \quad \text{and} \\
  \left| \frac{d^+}{dt} x_i(t+h)\right|
  &= |\dot{x}_i((t+h)+)|
  = \lim_{\e \to 0+} |\dot{x}_i(t+h+\e)|.
\end{align*}
Then, taking the limit $h \to 0+$, we obtain that $\frac{d^+}{dt}x_i$ is right continuous at $t$. By a similar argument it follows that $\frac{d^-}{dt}x_i$ is left continuous on $(0,T]$.
\end{proof}

%%%%%%%%%%%%%%%%%%%%%%%%%%%%%%%%%%%%%%%%%%%%%%%%%%%%%%%

\section{Global existence of solutions}
\label{sec3}
\setcounter{equation}{0}
In this section, we prove Theorem~\ref{exTh}. Before giving the proof in Section \ref{sec3:pf}, we first treat two preliminary sections; one on distance functions on $\Omega$ (Section \ref{sec3:Om}) and one on an approximate problem (Section \ref{sec3:Pk}). 

\subsection{Distance functions with respect to $\Omega$ and $\partial\Omega$}
\label{sec3:Om}
We define distance functions on $\R^m$ with respect to $\Omega$ and $\partial\Omega$. In general, for $K\subset\R^m$ and $\varepsilon>0$, we define
\[
\text{dist}(x,K):=\inf_{y\in K}|x-y|
\]
as the distance function with respect to $K$, and 
\[
N^{\varepsilon}(K):=\{x\in\R^m~;~\text{dist}(x,K)<\varepsilon\}
\]
as a neighborhood around $K$, including $K$ itself. Related to $\Omega$, we set
\begin{align*}
    d &: \R^m\to [0,\infty), & d(x) &:=\mathrm{dist}(x,\overline{\Omega}) ,\\[5pt]
    d_s &: \R^m\to \R, & d_s(x) &:=\left\{
    \begin{aligned}
      \mathrm{dist}(x,\partial\Omega)\quad(x\notin {\Omega}) &,\\
      -\mathrm{dist}(x,\partial\Omega)\quad(x\in {\overline \Omega}) &,
    \end{aligned}
    \right.
\end{align*}
where $d$ is the distance function with respect to $\Omega$, and $d_s$ is the \emph{signed} distance function with respect to $\partial \Omega$.

It is well-known that $d, d_s \in \mathrm{Lip}(\R^m)$ and that $d_s$ may be more regular close to $\partial \Omega$. In particular, since $\partial\Omega\in C^3$, there exists $\varepsilon>0$ such that $d\in C^3(\overline{N^{\varepsilon}(\partial\Omega)}\setminus\Omega)$ and $d_s\in C^3(\overline{N^{\varepsilon}(\partial\Omega)})$. We fix this $\varepsilon$ in the remainder of Section \ref{sec3}.

Next we build several functions from $d$ and $d_s$ for later use. First, we introduce
\begin{equation}\label{df}
d\nabla d : \overline{N^{\varepsilon}(\Omega)} \to \R^m, \qquad
(d\nabla d)(x):=\left\{
\begin{aligned}
& d(x)\nabla d(x)
& (x\in\overline{N^{\varepsilon}(\partial\Omega)}\setminus\overline{\Omega}),\\
& 0
& (x\in\overline{\Omega}).
\end{aligned}
\right.
\end{equation}
Since $\nabla d$ is not defined at $\partial \Omega$, we cannot split this function as the product of $d$ and $\nabla d$. In Proposition 2.5 in \cite{kimura2019particle} it is shown that the `combined' function $d\nabla d$ has the following regularity:

\begin{Prop}\label{d}
The function $d\nabla d$ defined in \eqref{df} satisfies
\[
d\nabla d\in \mathrm{Lip}(\overline{N^{\varepsilon}(\Omega)};\R^m).
\]
\end{Prop}

Next we construct the following extension $d_s^{\varepsilon}\in C^3(\R^m)$ of $d_s$:
\begin{equation}\label{depsilon}
d_s^{\varepsilon}(x) \left\{
\begin{aligned}
 & =d_s(x)
 & \big(x &\in N^{\varepsilon}(\partial \Omega)\big), \\
 & \ge \varepsilon 
 & \big(x &\in\R^m\setminus N^{\varepsilon}(\Omega)\big),\\
 & \le -\varepsilon 
 & \big(x &\in\Omega\setminus N^{\varepsilon}(\partial \Omega)\big).
\end{aligned}
\right.
\end{equation}
In what follows, we denote for a function $f : \R^m \to \R$ by $\nabla^2 f$ the Hessian matrix, and by $\nabla^3 f$ the tensor
\[ [\nabla^3 f(x)]_{ijk} = \partial_{x_i} \partial_{x_j} \partial_{x_k} f(x). \]

\begin{Lem}\label{Weingarten}
$\nabla^2 d_s^{\varepsilon} \nabla d_s^{\varepsilon}$ vanishes on $N^{\varepsilon}(\partial\Omega)$, i.e. 
\[
\nabla^2 d_s^{\varepsilon}(x) \nabla d_s^{\varepsilon}(x) = 0
\qquad ( x\in N^{\varepsilon}(\partial\Omega) ).
\]
\end{Lem}
\begin{proof}
Let $x\in N^{\varepsilon}(\partial\Omega)$. Since $\nabla d_s^\e$ is constant along any line which pierces through $\partial \Omega$ in perpendicular direction,
\begin{equation*}
  0 
  = \frac d{dt} \nabla d_s^\e (x + t \nabla d_s^\e(x)) \big|_{t=0}
  = \nabla^2 d_s^\e (x) \nabla d_s^\e (x).
\end{equation*}
\end{proof}

The third set of functions that we introduce is related to the change of variables illustrated in Figure \ref{Pmap}.

\begin{figure}[htbp]
  \begin{center} 
    \includegraphics[width=90mm]{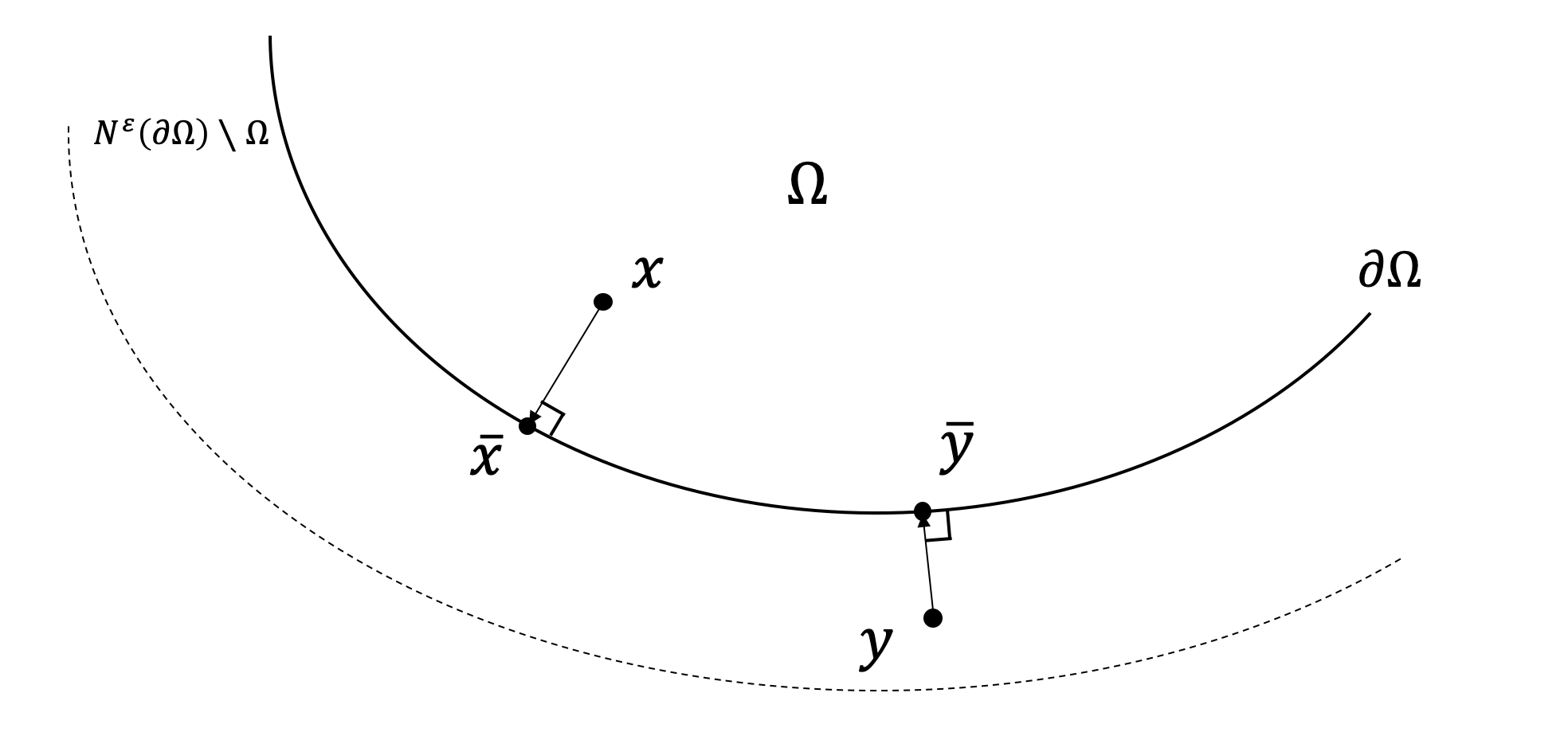}
  \end{center}
  \caption{Sketch of the variable transformation \eqref{xbar:r} from $x$ to $(\bar x, r)$.}
  \label{Pmap}
\end{figure}

\begin{Prop}\label{Phi}
For $\partial\Omega \in C^l~(l\ge3)$, the function 
\[
\Phi : \partial\Omega\times (-\e,\e)\to N^{\e}(\partial\Omega),
\quad \Phi(\bar x, r) = \bar x + r \nu(\bar x)
\]
is a $C^{l-1}$ diffeomorphism. Moreover, the inverse is given by
\begin{equation}\label{xbar:r}
\Phi^{-1}(x) = (P(x), d_s(x)) \qquad(x\in N^\e(\partial\Omega))
\end{equation}
where $P\in C^{l-1}(N^\e(\partial\Omega); \R^m)$ is given by
\[
  P(x) = x - d_s(x)\nabla d_s(x) \in\partial\Omega.
\]  
\end{Prop}

We refer to Proposition 3.6 in \cite{kimura2008geometry} and \cite{gilbarg2015elliptic} for the proof of Proposition \ref{Phi}.

We interpret $\bar x = P(x)$ as the orthogonal projection of $x\in N^\e(\partial\Omega)$ onto $\partial\Omega$. We note from Lemma \ref{Weingarten} and $|\nabla d_s| = 1$ that 
\begin{equation}\label{Wmap}
  \nabla P \nabla d_s
  = (I - \nabla d_s \otimes \nabla d_s  - d_s \nabla^2 d_s ) \nabla d_s
  = \nabla d_s - \nabla d_s - 0
  = 0 \qquad \text{on } N^{\varepsilon}(\partial\Omega).
\end{equation}

%%%%%%%%%%%%%%%%%%%%%%%%%%%%%%%%%%%%%%%%%%%%%%%%%%%%%%%

\subsection{The approximation of Problem~\ref{ProbP}}
\label{sec3:Pk}

We construct an approximation of Problem~\ref{ProbP} by replacing the reflection rule and the condition that $x_i \in \overline \Omega$ by a force field outside of $\Omega$ which pushes a particle back inside $\Omega$. This force field is given by $- k d\nabla d : \R^m \to \R^m$, where $k \in \N$ is a (large) parameter which dictates the strength of the force. Since particles may leave $\Omega$, we need to extend the domain of $F_i$ beyond $\overline \Omega^n$. By \eqref{1}, we may extend it to $(\R^m)^n$ as a continuous function. In the following, we take any such extension, and denote it simply by
\begin{equation}\label{conti}
F_i\in C([0,T]\times (\R^m)^n;\R^m)\quad(i=1,\ldots,n).
\end{equation}
\begin{Prob}\label{AP}
Let $k\in \mathbb{N}$ and $T>0$. For $i=1,\ldots,n$, let $F_i$ be as in \eqref{conti}, and let $x_i^{\circ}\in\Omega$, $v_i^{\circ}\in \R^m$. Find a solution $X^k(t)=\{x_i^k(t)\}_{i=1}^n$ to
\begin{equation} \label{AP:eq}
\left\{
\begin{aligned}
\ddot{x}_i^k(t) &= F_i(t, X^k(t)) - k(d\nabla d)(x_i^k(t))\qquad (t\in(0,T)),\\
x_i^k(0) &= x_i^{\circ}, \\
\frac{d x_i^k}{dt}(0) &= v_i^{\circ}
\end{aligned}
\right.
\end{equation}
for $i=1,\ldots,n$.
\end{Prob}

Figure \ref{test} illustrates how solutions to Problem \ref{AP} approximate the reflection in a more regular fashion.
\begin{figure}[htbp]
  \begin{center} 
    \includegraphics[width=140mm]{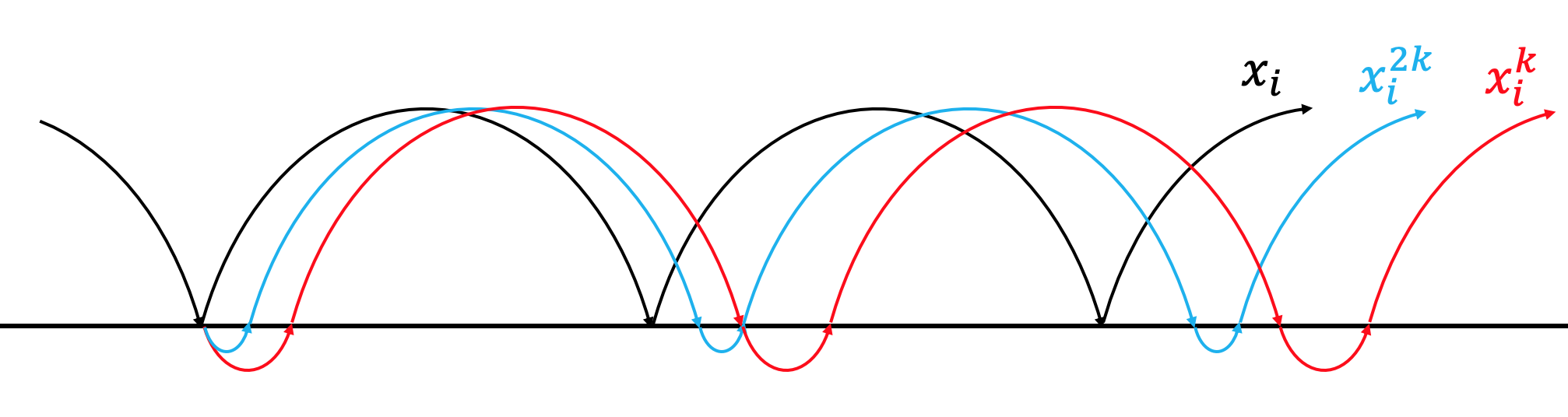}
  \end{center}
  \caption{Sketch of the trajectory of a solution $x_i$ to Problem \ref{ProbP} and two trajectories of solutions to the approximate Problem \ref{AP}.}
  \label{test}
\end{figure}

For a continuous solution $X^k$ to Problem \ref{AP}, we set
\begin{equation*} \label{Iki}
  I^{i,k} := \left\{ t\in [0,T]~;~ x_i^k(t) \notin \overline{\Omega}\right\}
  \qquad (i = 1,\ldots,n)
\end{equation*}
as the set of times at which particle $x_i^k$ is outside of $\overline \Omega$.
The following lemma provides the global existence of classical solutions to Problem~\ref{AP}. It also lists several bounds on such solutions.

\begin{Lem}\label{l:xk}
For all $k$ large enough, there exists a classical solution $X^k \in C^2([0,T]; \R^m)$ to Problem~\ref{AP}. Moreover, there exists a constant $C > 0$ such that for all $k$ large enough and for any classical solution $X^k$:
\begin{enumerate}[(i)] 
  \item \label{l:xk:d} $\| d(x_i^k) \|_{C([0,T])} \leq C / \sqrt k$ for $i = 1,\ldots,n$,
  \item \label{l:xk:xd} $\| \dot X^k \|_{C([0,T])} \leq C$,
  \item \label{l:xk:xddL1} $\| \ddot X^k \|_{L^1(0,T)} \leq C$, 
  \item \label{l:xk:kdndL1} $\| (d \nabla d)(x_i^k) \|_{L^1(0,T)} \leq C/k$ for $i = 1,\ldots,n$,
  \item \label{l:xk:xddCI} $\| \ddot x_i^k \|_{C([0,T] \setminus I^{i,k})} \leq C$ for $i = 1,\ldots,n$.
\end{enumerate}
\end{Lem}

\begin{proof}
Setting $v_i^k(t):=\dot{x}_i^k(t)$ for all $i\in \{1,\ldots,n\}$, we rewrite \eqref{AP:eq} as
\begin{equation}\label{change}
\left\{
\begin{aligned}
\frac{d}{dt}\begin{bmatrix} v_i^k(t) \\ x_i^k(t) \end{bmatrix}  &= \begin{bmatrix} F_i(t, X^k(t)) - k(d\nabla d)(x_i^k(t)) \\ v_i^k(t) \end{bmatrix},\\
\begin{bmatrix} v_i^k(0) \\ x_i^k(0) \end{bmatrix} &= \begin{bmatrix} v_i^{\circ} \\ x_i^{\circ} \end{bmatrix}.
\end{aligned}
\right.
\end{equation}
From \eqref{1} and Proposition~\ref{df} it follows that the right-hand side is continuous in $t$ and $X$ as long as  $X \in (\overline{N^{\varepsilon}(\Omega)})^n$. Hence, there exists some $T_0>0$ such that \eqref{change} has a classical solution $(x_i^k, v_i^k) \in C^1([0,T_0]; (\R^{m})^2)$ with $x_i^k(t)\in N^{\varepsilon}(\Omega)$ for all $t\in [0, T_0)$ and all $i\in \{1,\ldots,n\}$. Since $\dot{x}_i^k = v_i^k \in C^1([0,T_0])$, $X^k \in C^2([0,T_0]; (\R^m)^n)$, and thus Problem~\ref{AP} has a local classical solution. 

To show that any local solution $X^k$ can be extended to a global solution, we need to show that the right-hand side of the ODE in \eqref{change} is continuous along the evolution, i.e., $x_i^k(t)\in N^{\varepsilon}(\Omega)$ for all $t\in [0, T]$ and all $i\in \{1,\ldots,n\}$. Suppose that this is not the case. Let $T_0 > 0$ be the first time at which $d(x_i^k(T_0)) = \e$ for some $i \in \{1,\ldots,n\}$. Then, for $t\in (0,T_0)$ we calculate 
\[
\begin{aligned}
\frac{1}{2}\frac{d}{dt}|\dot{x}_i^k(t)|^2 &=\dot{x}_i^k(t)\cdot \ddot{x}_i^k(t)\\
&=\dot{x}_i^k(t)\cdot F_i(t, X^k(t)) - \dot{x}_i^k(t) k(d\nabla d)(x_i^k(t))\\
&=\dot{x}_i^k(t)\cdot F_i(t, X^k(t)) -\frac{1}{2} k \frac{d}{dt}d(x_i^k(t))^2.
\end{aligned}
\]
Integration over $(0,T_0)$ yields
\begin{equation}\label{l:pf:1}
  \begin{aligned} 
\frac{1}{2} |\dot{x}_i^k(t)|^2+\frac{k}{2}d(x_i^k(t))^2 
&= \frac{1}{2}|v_i^{\circ}|^2+ \frac{k}{2}d(x_i^\circ)^2 +\int_0^t \dot{x}_i^k(s)\cdot F_i(s, X^k(s)) ds\\ 
&=\frac{1}{2}|v_i^{\circ}|^2 + \int_0^t\dot{x}_i^k(s)\cdot F_i(s, X^k(s)) ds\\ 
&\le \frac{1}{2}|v_i^{\circ}|^2 + \int_0^t\frac{|F_i(s, X^k(s))|^2}{2}ds + \int_0^t\frac{|\dot{x}_i^k(s)|^2}{2}ds \\
&\le \underbrace{ \frac{1}{2}|v_i^{\circ}|^2 + \frac T2 \| F_i \|_\infty^2 }_{= C} + \int_0^t\frac{|\dot{x}_i^k(s)|^2}{2}ds.
\end{aligned}  
\end{equation}
By removing the non-negative contribution from the second term in the left-hand side, an application of Gronwall's Lemma yields
\[
\frac{1}{2}|\dot{x}_i^k(t)|^2 
\le C e^t 
\le C e^T
\]
for all $t \in (0,T_0)$. Inserting this uniform bound in \eqref{l:pf:1}, we obtain
\[
\frac{k}{2} d(x_i^k(t))^2
\le C,
\]
where here and henceforth we allow the constant $C$ to change from line to line. Hence,
\[
d(x_i^k(t)) \le \frac{C}{\sqrt k}
\]
for all $t \in (0,T_0)$ and all $i\in \{1,\ldots,n\}$. In particular, if $k > 4{C^2}/{\varepsilon^2} $, then $d(x_i^k(t)) < \e/2$ for all $t \in (0, T_0)$. By the continuity of $x_i^k$, we then obtain $d(x_i^k(T_0)) \leq \varepsilon/2$, which contradicts with our assumption on $T_0$. Hence, the solution $X^k$ extends to $T$.

Next we prove Properties \eqref{l:xk:d}--\eqref{l:xk:xddCI}. Properties \eqref{l:xk:d} and \eqref{l:xk:xd} are already established in the argument above. Property \eqref{l:xk:xddCI} follows from the observation from \eqref{AP:eq} that 
\[
  |\ddot x_i^k(t)| = |F_i(t ,X^k(t))| \leq \|F_i\|_{\infty}
\]  
for all $t \notin I^{i,k}$ and all $i\in\{1,\ldots,n\}$.
  
Finally, we prove Properties \eqref{l:xk:xddL1} and \eqref{l:xk:kdndL1}. First, for all $i\in\{1,\ldots,n\}$, we estimate
\begin{equation}\label{BVbdd:0}
\begin{aligned}
\int_0^T|\ddot{x}_i^k(t)|dt& =\int_0^T |F_i(t, X^k(t)) - k(d\nabla d)(x_i^k(t))|dt\\
& \le \int_0^T |F_i(t, X^k(t))|dt + \int_{0}^T |k(d\nabla d)(x_i^k(t))|dt.
\end{aligned}
\end{equation}
To continue with estimating the second term, note that for $x\in N^{\varepsilon}(\Omega)$ at least one of the following holds: 
\begin{enumerate}
  \item $x \in \overline \Omega$ and thus $(d\nabla d)(x) = 0$, or
  \item $x \in N^\e (\partial \Omega) \setminus \overline \Omega$ and thus $(d\nabla d)(x) = d(x) \nabla d(x) = d(x) \nabla d_s^\e(x)$.
\end{enumerate}
In both cases, $|(d\nabla d)(x)|=(d\nabla d)(x)\cdot \nabla d_s^{\varepsilon}(x)$, and thus
\begin{equation}\label{BVbdd}
\begin{aligned}
\int_0^T \left| k (d \nabla d)(x_i^k(t))\right|dt
&=\int_0^T k (d \nabla d)(x_i^k(t)) \cdot \nabla d_s^{\varepsilon}(x_i^k(t))dt\\
&=\int_0^T (F_i(t, X^k(t)) - \ddot{x}_i^k(t)) \cdot \nabla d_s^{\varepsilon}(x_i^k(t))dt\\
&\le C T \|F_i\|_{\infty} - \big[\dot{x}_i^k(t)\cdot \nabla d_s^{\varepsilon}(x_i^k(t)) \big]_0^T+\int_0^T \dot{x}_i^k(t)\cdot\left(\nabla^2 d_s^{\varepsilon}(x_i^k(t))\dot{x}_i^k(t)\right)dt\\
&\le C T \|F_i\|_{\infty} + |\dot{x}_i^k(0)|+|\dot{x}_i^k(T)|+\|\nabla^2 d_s^{\varepsilon}\|_{\infty}\int_0^T |\dot{x}_i^k(t)|^2 dt,
\end{aligned}
\end{equation}
which is bounded uniformly in $k$ by Property \eqref{l:xk:xd}. This proves Property \eqref{l:xk:kdndL1}. Using Property \eqref{l:xk:kdndL1} and applying again the uniform bound on $F_i$, we obtain Propery \eqref{l:xk:xddL1} from \eqref{BVbdd:0}.
\end{proof}

%%%%%%%%%%%%%%%%%%%%%%%%%%%%%%%%%%%%%%%%%%%%%%%%%%%%%%%%%%%%%

\subsection{Proof of existence of solutions}
\label{sec3:pf}

\paragraph{Construction of $X$ as in Definition \ref{Def}(i)} First we construct a candidate solution $X$. Given $T, x_i^\circ, v_i^\circ$ as in Problem \ref{ProbP}, let $X^k$ be a corresponding solution to Problem~\ref{AP} for each $k$ large enough. The following Lemma specifies compactness properties for the sequence $\{X^k\}$, and shows that any limit point satisfies Definition \ref{Def}(i). 
\begin{Lem}\label{Lem1}
Let $\{X^k\}$ be as above. Then, there exists an $X$ which satisfies Definition \ref{Def}(i) such that
 \begin{equation}\label{converge}
  \left\{
  \begin{aligned}
    X^k\to X &\qquad\text{in}~C^0([0,T];(\R^m)^n),\\
    \dot{X}^k\to \dot{X} &\qquad\text{in}~L^p(0,T;(\R^m)^n)\qquad(p\in [1,\infty))
    \end{aligned}
  \right.
  \end{equation}
  as $k\to\infty$.
\end{Lem}

\begin{proof}
It is not restrictive to take $i \in \{1,\ldots,n\}$ arbitrarily, and prove the statements of Lemma \ref{Lem1} for $x_i^k$. In the following, we will take several subsequences without relabeling the index.

We start with proving the first statement in \eqref{converge}. By Lemma~\ref{l:xk}\eqref{l:xk:d},\eqref{l:xk:xd}, for $k$ large enough $\{x^k\}$ is bounded and equicontinuous such that Ascoli-Arzel{\`a}'s Theorem applies. Hence, there exists a uniformly convergent subsequence of $\{x_i^k\}$ with limit point $x_i \in C([0,T], \R^m)$. 

Next we show the auxiliary statement that
\begin{equation} \label{pf:dxk:Linf}
  \dot{x}_i^k \stackrel*\rightharpoonup \dot{x}_i \qquad \text{in}~L^{\infty}(0,T;\R^m)
\end{equation}
along a further subsequence as $k \to \infty$.
From the uniform bound on $\| \dot{x}_i^k \|_\infty$ (see Lemma \ref{l:xk}\eqref{l:xk:xd}), there exists a further subsequence of $\{x_i^k\}$ along which $\dot{x}_i^k \stackrel*\rightharpoonup y$ in $L^{\infty}$ as $k \to \infty$ for some $y\in L^{\infty}(0,T;\R^m)$. To show that $y=\dot{x}_i$, we characterize $y$ as the weak derivative of $x_i$. For any test function $\psi \in C_0^\infty((0,T); \R^m)$, there holds
\[
\begin{aligned}
\int_0^Ty(t)\cdot \psi(t)dt &= \lim_{k \to\infty}\int_0^T \dot{x}_i^{k}(t)\cdot\psi(t)dt=-\lim_{k \to\infty}\int_0^T x_i^{k}(t)\cdot\psi '(t)dt\\
&=-\int_0^T x_i(t)\cdot\psi '(t)dt.
\end{aligned}
\]
Next we show that $x_i$ satisfies Definition \ref{Def}(i). From the two convergence results established above, we observe that $x_i \in W^{1,\infty}(0,T; \R^m)$. From Lemma \ref{l:xk}\eqref{l:xk:d} it follows that $x_i(t) \in \overline \Omega$ for all $t$. Hence, it is left to show that $\dot{x}_i \in BV(0,T; \R^m)$. Since $\ddot{x}_i^k$ is bounded in $L^1(0,T; \R^m)$ (see Lemma \ref{l:xk}\eqref{l:xk:xddL1}), it follows by Theorem 5.5 form \cite{evans2015measure} that $\dot x_i^k$ converges strongly along a further subsequence in $L^1(0,T; \R^m)$ to a limit $z \in BV(0,T; \R^m)$. By \eqref{pf:dxk:Linf} and the uniqueness of limits, we obtain $\dot{x}_i = z \in BV(0,T; \R^m)$. 

Finally, we use again the strong convergence of $\dot{x}_i^k$ in $L^1$ to prove the second statement in \eqref{converge}. Since $\dot{x}_i^k, \dot{x}_i \in L^\infty$ (see \eqref{pf:dxk:Linf}), 
\[ \|\dot{x}_i^k-\dot{x}_i\|_{L^p}^p
= \int_0^T |\dot{x}_i^k-\dot{x}_i|^{p-1} |\dot{x}_i^k-\dot{x}_i| dt
\le ( \| \dot{x}_i^k \|_\infty + \| \dot{x}_i \|_\infty )^{p-1} \|\dot{x}_i^k-\dot{x}_i\|_{L^1}
\le C \|\dot{x}_i^k-\dot{x}_i\|_{L^1} \to 0. \]
\end{proof}

\paragraph{$X$ satisfies Definition \ref{Def}(ii)} Let $X$ be given by Lemma~\ref{Lem1}, and take $i \in \{1, \ldots, n\}$ arbitrary. In the following we prove that $x_i$ satisfies Definition~\ref{Def}(ii) for some $\rho_i \in \mathcal{M}_+((0,T))$ with $\supp \rho_i \subset I_{\partial \Omega}^i$ (defined in Definition~\ref{Def}(ii)). With this aim, let $\psi\in C_0^{\infty}(0,T;\R^m)$ be a test function, and note that $x_i^k$ satisfies
\[
  \int_0^T (-\ddot{x}_i^k + F_i(t, X^k(t)))\cdot \psi(t) dt
  =  \int_0^T k(d\nabla d)(x_i^k(t))\cdot\psi(t)dt 
  =: A_i^k.
\]
With Lemma \ref{Lem1} and the continuity of $F_i$ we pass to the limit $k \to \infty$ in the left-hand side. This yields
\[
\begin{aligned}
\lim_{k\to\infty} \int_0^T ( - \ddot{x}_i^k + F_i(t, X^k(t)))\cdot \psi(t) dt &=\lim_{k\to\infty}\left(\int_0^T(\dot{x}_i^k(t)\cdot\dot{\psi}(t) + F_i(t, X^k(t))\cdot\psi(t))dt\right)\\
&=\int_0^T(\dot{x}_i(t)\cdot\dot{\psi}(t) + F_i(t, X(t))\cdot\psi(t))dt.
\end{aligned}
\]
Hence, it is left to show that 
\begin{equation} \label{Ak:purpose}
  \lim_{k\to\infty} A_i^k = \int_0^T \nu(x_i(t)) \cdot \psi(t) d\rho_i(t)
\end{equation}
for some $\rho_i \in \mathcal{M}_+((0,T))$ with $\supp \rho_i \subset I_{\partial \Omega}^i$.

To prove \eqref{Ak:purpose}, we assume that $k$ is large enough. Then, since $d(x_i^k(t)) < \e$ for all $t \in [0,T]$, either $d(x_i^k(t)) = 0$ or $\nabla d(x_i^k(t)) = \nabla d_s^\e(x_i^k(t))$.
Hence,
\[
A_i^k =\int_0^T \underbrace{ k d(x_i^k(t)) }_{=: \rho_i^k(t)} \underbrace{ \nabla d_s^{\varepsilon}(x_i^k(t))\cdot\psi(t) }_{=: f_i^k(t)} dt.
\]
Note from \eqref{converge} that 
\begin{equation} \label{fk:conv}
  f_i^k \to \nabla d_s^{\varepsilon}(x_i)\cdot\psi \quad \text{in } C([0,T])
\end{equation} 
as $k \to \infty$ along a subsequence.
Since $\rho_i^k(t) \geq 0$, we may interpret $\rho_i^k \in \mathcal{M}_+((0,T))$. Then, by Lemma \ref{l:xk}\eqref{l:xk:kdndL1},
\[
\begin{aligned}
\int_0^T  \rho_i^k(t) dt &=\int_0^T k(d\nabla d)(x_i^k(t))\cdot \nabla d_s^\e (x_i^k(t)) dt
\le C \| \nabla d_s^\e \|_\infty < \infty.
\end{aligned}
\] 
Since this bound is uniform in $k$, there exists a $\rho_i \in \mathcal{M}_+((0,T))$ such that $\rho_i^k \stackrel*\rightharpoonup \rho_i$ in $\mathcal{M}_+((0,T))$ along a subsequence as $k \to \infty$. Together with \eqref{fk:conv}, we then obtain
\[
  \lim_{k\to\infty} A_i^k = \int_0^T \nabla d_s^\e(x_i(t)) \cdot \psi(t) d\rho_i(t).
\]

To conclude \eqref{Ak:purpose}, it is left to show that 
\begin{equation} \label{supprho}
  \supp \rho_i \subset I_{\partial\Omega}^i,
\end{equation}
in which case we may replace $\nabla d_s^\e(x_i(t))$ by $\nu(x_i(t))$.
We prove \eqref{supprho} by contradiction. Suppose there exists a $t_0\in \supp \rho_i$ with $t_0 \notin I_{\partial \Omega}^i$. Since $(I_{\partial \Omega}^i)^c$ is open, there exists a $\delta>0$ such that $\overline{B_{\delta}(t_0)} \subset (I_{\partial \Omega}^i)^c$ with $\sup_{t\in B_{\delta}(t_0)} d_s(x_i(t)) < 0$. By using \eqref{converge}, for all $k$ large enough, we have $\sup_{t\in B_{\delta}(t_0)} d_s(x_i^k(t)) < 0$ for all $t\in B_{\delta}(t_0)$. Hence, $\rho_i^k(t) = k d(x_i^k(t)) = 0$ for all $t\in B_{\delta}(t_0)$, and thus
\[ 
\supp \rho_i^k \cap B_{\delta}(t_0) = \emptyset
\] 
for all $k$ large enough. Passing to $k\to\infty$, we get  $\supp \rho_i \cap B_{\delta}(t_0) = \emptyset$, which contradicts with $t_0 \in \supp \rho_i$. This concludes the proof for the statement that $X$ satisfies Definition \ref{Def}(ii).

\paragraph{$X$ satisfies Definition \ref{Def}(iii)} 
For arbitrary $i \in \{1,\ldots,n\}$ and $t_0 \in I_{\partial \Omega}^i$ (defined in Definition \ref{Def}(ii)), we need to show that $x_i$ satisfies the reflection rule \eqref{reflection} at $t = t_0$. We recall from the proof of Proposition \ref{speed} that the reflection rule induces a jump discontinuity only for the normal component of the velocity, and that the tangential component remains continuous. To prove that $\dot x_i$ has these properties, we employ the change of variable introduce in Proposition \ref{Phi} with related functions $\Phi$, $P$ and $d_s$. Changing variables requires $x_i(t) \in N^\e(\partial \Omega)$. Since $\| \dot x_i \|_\infty < \infty$ and $x_i(t_0) \in \partial \Omega$, there exists a $\delta > 0$ such that $x_i(t) \in N^\e(\partial \Omega)$ for all $t \in B(t_0, \delta)$. By the uniform convergence of $x_i^k$ to $x_i$ (see \eqref{converge}), we may take $\delta$ smaller if needed such that for all $k$ large enough $x_i^k(t) \in N^\e(\partial \Omega)$ for all $t \in B(t_0, \delta)$. 

Next we apply the change of variables on $x_i(t)$ and $x_i^k(t)$ for any $t \in B(t_0, \delta)$. We set
\begin{align*}
  \bar{x}_i(t) \, &:= P(x_i(t)) & r_i(t) \, &:= d_s(x_i(t)), \\
  \bar{x}_i^k(t) \, &:= P(x_i^k(t)) & r_i^k(t) \, &:= d_s(x_i^k(t)).
\end{align*}
Note that
\[
  x_i(t) = \Phi( \bar{x}_i(t), r_i(t) ) = \bar{x}_i(t) + r_i(t) \nu (\bar{x}_i(t))
\]
and
\[
\dot{\bar{x}}_i(t)= \nabla P(x_i(t)) \dot{x}_i(t), \quad
\dot r_i(t) = \nabla d_s(x_i(t)) \cdot \dot{x}_i(t)
\qquad 
\text{for a.e.~} t\in B(t_0, \delta),
\]
and that similar expressions hold for $x_i^k(t)$.
Since $x_i \in C^0([0,T];\R^m)$ and $\dot{x}_i \in BV(0,T;\R^m)$, we have $\dot{\bar{x}}_i \in BV(B(t_0, \delta);\R^m)$ and $\dot{r}_i \in BV(B(t_0, \delta))$.

Since $\dot{\bar{x}}_i(t)$ is tangent to $\Omega$ at $\bar{x}_i(t)$, we expect that its second derivative has no jump discontinuity, and thus that $\dot{\bar{x}}_i$ is more regular than a typical BV function. We prove this by approximation from the approximate Problem \ref{AP}. We obtain from the convergence of $x_i^k$ to $x_i$ in \eqref{converge} and the regularity of $P$ that $\bar{x}_i^k \to \bar{x}_i$ in $C^0( [0,T]; \R^m)$ and $\dot{\bar{x}}_i^k \to \dot{\bar{x}}_i$ in $L^p(0,T; \R^m)$ as $k \to \infty$. Using the ODE \eqref{AP:eq} for $x_i^k$ and \eqref{Wmap}, we compute
\[
\begin{aligned}
\ddot{\bar{x}}_i^k(t)
&=\nabla^2 P(x_i^k(t)) \dot{x}_i^k(t) \dot{x}_i^k(t) + \nabla P(x_i^k(t)) \ddot{x}_i^k\\
&= \nabla^2 P(x_i^k(t)) \dot{x}_i^k(t) \dot{x}_i^k(t) + \nabla P(x_i^k(t)) \big(F_i(t,X^k(t))  - k(d\nabla d)(x_i^k(t)) \big)\\
&=  \nabla^2 P(x_i^k(t)) \dot{x}_i^k(t) \dot{x}_i^k(t) + \nabla P(x_i^k(t)) F_i(t,X^k(t)).
\end{aligned}
\]
Since the right-hand side is bounded in $B(t_0, \delta)$, we obtain that $\ddot{\bar{x}}_i^k \rightharpoonup y$ weakly-* in $L^{\infty}(B(t_0, \delta);\R^m)$ along a subsequence as $k \to \infty$. We characterize $y$ as the weak derivative of $\dot{\bar x}_i$; for any test function $\phi(t)\in C^{\infty}_0(B(t_0, \delta))$, we have
\[
\begin{aligned}
\int_{B(t_0, \delta)} y \cdot \phi(t) dt 
&= \lim_{k\to\infty} \int_{B(t_0, \delta)} \ddot{\bar{x}}_i^k(t)\cdot\phi(t)dt\\
&=-\lim_{k\to\infty} \int_{B(t_0, \delta)} \dot{\bar{x}}_i^k(t)\cdot\phi '(t)dt\\
&=-\int_{B(t_0, \delta)} \dot{\bar{x}}_i(t)\cdot\phi '(t) dt.
\end{aligned}
\]
Hence, $\ddot{\bar{x}}_i = y \in L^\infty(B(t_0, \delta))$, and thus $\dot{\bar{x}}_i(t_0+)=\dot{\bar{x}}_i(t_0-)$. 

Next we use this result to narrow down the reflection rule \eqref{reflection}. Recalling that $r_i(t_0) = 0$, we obtain
\[
\begin{aligned}
\frac{d^+}{dt}x_i(t_0) - \frac{d^-}{dt}x_i(t_0)
&=\frac{d^+}{dt}\left[\bar{x}_i + r_i \nabla d_s(\bar x_i)\right](t_0) - \frac{d^-}{dt}\left[\bar{x}_i + r_i \nabla d_s(\bar x_i)\right](t_0)\\
&= \left( \frac{d^+}{dt}r_i(t_0) - \frac{d^-}{dt}r_i(t_0) \right) \nabla d_s(\bar x_i(t_0)).
\end{aligned}
\]
Since $x_i(t_0) \in \partial \Omega$, we have $\nabla d_s(\bar x_i(t_0)) = \nabla d_s(x_i(t_0)) = \nu(x_i(t_0))$, and thus the reflection rule \eqref{reflection} holds if 
\begin{equation} \label{r:equal}
  \dot r_i(t_0+) = -\dot r_i(t_0-).
\end{equation}

The remainder of the proof of Theorem~\ref{exTh} focusses on proving \eqref{r:equal}. We first prove \eqref{r:equal} under the additional assumption that 
\[
  \alpha := \dot{r}_i(t_0-) > 0.
\] 
The following proof is inspired by Figure \ref{test}; we first give an overview. First, we show that the approximate solution $r_i^k:=d_s(x_i^k)$ leaves $\Omega$ at some time $a^k$ which is approximately equal to $t_0$, i.e.
\begin{equation} \label{alam1}
  a^k \to t_0 \qquad \text{as } k \to \infty,
\end{equation}
and that the speed of impact is approximately equal to that of $r_i$, i.e.
\begin{equation} \label{alam2}
  \dot{r}_i^k(a^k) \to \dot r_i(t_0-) \qquad \text{as } k \to \infty.
\end{equation}
Then, we show that $x_i^k \notin \overline \Omega$ on a small interval $(a^k, b^k)$, i.e.
\begin{equation} \label{alam3}
  b^k - a^k \to 0 \qquad \text{as } k \to \infty,
\end{equation}
and that $x_i^k$ enters $\Omega$ at time 
\[
  b^k := \inf \{t > a_k \ ; \ r_i^k(t) \leq 0 \}. 
\] 
with normal speed approximately equal to the negative of the speed of impact, i.e.
\begin{equation} \label{alam4}
  \dot{r}_i^k(a^k) + \dot{r}_i^k(b^k) \to 0 \qquad \text{as } k \to \infty.
\end{equation}
To conclude \eqref{r:equal}, it is then left to show that 
\begin{equation} \label{alam5}
  \dot r_i^k(b^k) \to \dot r_i(t_0+) \qquad \text{as } k \to \infty.
\end{equation} 

Next we prove \eqref{alam1}--\eqref{alam5}. Since $\dot{r}_i$ is left-continuous at $t_0$, we have that for any $\e_1 \in (0,\alpha/2)$ there exists $\delta_1 \in (0,\delta]$ such that 
\[
  \big| \dot{r}_i(t_0 - s) - \alpha \big| \leq \e_1 \qquad (s \in (0,\delta_1)).
\]
Hence, for any $s \in (0,\delta_1)$,
\begin{equation} \label{alam1:pf:1}
 r_i(t_0 - s) 
  = -\int_{t_0 - s}^{t_0} \dot{r}_i(t) dt
  \leq -s(\alpha - \e_1) < 0.
\end{equation}

Let 
\[
  a^k := \inf \{t > t_0 - \delta_1 \ ; \ r_i^k(t) = 0 \}. 
\]  
Relying on the uniform convergence of $r_i^k$ to $r_i$ (see \eqref{converge}), we obtain from \eqref{alam1:pf:1} that
\[
  \liminf_{k \to \infty} a^k \geq t_0.
\]
Then, to prove \eqref{alam1}, it is left to show that $\limsup_k a^k \le t_0$. We reason by contradiction. Suppose that there exists $\delta_2 > 0$ such that $a^k \geq t_0 + \delta_2$ for a subsequence of $k$ (not relabeled). Then, for any $\delta_3 \in (0, \delta_1 \wedge \delta_2]$ and any $t \in B(t_0, \delta_3)$
\begin{equation} \label{alam1:pf:2}
  |\dot r^k(t) - \alpha|
  \leq |\dot{r}_i^k(t) - \dot{r}_i^k(t_0 - \delta_3)| + |\dot{r}_i^k(t_0 - \delta_3) - \dot{r}_i(t_0 - \delta_3)| + |\dot{r}_i(t_0 - \delta_3) - \alpha|.
\end{equation}
The third term is bounded by $\e_1$ for a.e.\ choice of $\delta_3$. For the first term, we note from the definition of $a_k$ and Lemma \ref{l:xk}\eqref{l:xk:xddCI} that $\dot{r}_i^k$ is Lipschitz continuous on $\overline{B(t_0, \delta_3)}$ with Lipschitz constant $L$ independent of $k$ and $\delta_3$. Hence,
\[
  |\dot{r}_i^k(t) - \dot{r}_i^k(t_0 - \delta_3)| \leq L |t - (t_0 - \delta_3)| \leq 2L \delta_3.
\]
Finally, by \eqref{converge}, the second term in \eqref{alam1:pf:2} converges pointwise for a.e.\ $\delta_3$ as $k \to \infty$. Hence, by choosing $\delta_3$ small enough and away from a nullset, we obtain from \eqref{alam1:pf:2} that
\begin{equation} \label{alam1:pf:3}
|\dot{r}_i^k(t) - \alpha| \leq 3 \e_1 \qquad ( t \in B(t_0, \delta_3))
\end{equation}
for all $k$ large enough. Then, taking $\e_1 < \alpha/6$, we obtain
\[
\dot{r}_i^k(t) \geq \alpha/2 \qquad  (t \in B(t_0, \delta_3))
\]
for all $k$ large enough. Since $r_i^k(t_0) \to r_i(t_0) = 0$ as $k \to \infty$, we conclude that $r_i^k(t_0 + \delta_3) > 0$ for all $k$ large enough. Thus, $a_k < t_0 + \delta_3 \leq t_0 + \delta_2$ for all $k$ large enough, which contradicts our previous assumption on $a_k$. This concludes the proof of \eqref{alam1}.

To prove \eqref{alam2}, we reuse \eqref{alam1:pf:3}. Indeed, since by \eqref{alam2} $a_k \in B(t_0, \delta_3)$ for all $k$ large enough, \eqref{alam1:pf:3} holds for all $t \in [t_0 - \delta_3, a_k]$, and thus in particular for $t = a_k$. Since $\e_1$ is arbitrary, we conclude \eqref{alam2}.

Next we prove \eqref{alam3}. Using the ODE for $x_i^k$ on $(a^k, b^k)$ (i.e., a time interval during which $x_i^k \notin \overline \Omega$), we get from the uniform bounds on $F_i$, $\nabla d_s$, $\nabla^2 d_s$ and $\dot{x}_i^k$ that
\begin{equation} \label{alam3:pf:1}
\begin{aligned}
\ddot{r}_i^k(t) &= \frac{d^2}{dt^2} d_s(x_i^k(t)) = \ddot x_i^k(t) \cdot \nabla d_s(x_i^k(t)) + \left(\nabla^2 d_s(x_i^k(t))\dot{x}_i^k(t)\right)\cdot\dot{x}_i^k(t)
\\
&= -k d_s(x_i^k(t)) + F_i(t, X^k(t))\cdot \nabla d_s(x_i^k(t)) + \left(\nabla^2 d_s(x_i^k(t))\dot{x}_i^k(t)\right)\cdot\dot{x}_i^k(t)
\leq -k r_i^k(t) + C
\end{aligned}
\end{equation}
for all $t \in (a^k, b^k)$, where $C > 0$ is a constant independent of $k$ and $t$. We compare this differential inequality with the differential equation
\[
\left\{
\begin{aligned}
  \ddot{R}^k(t) &= -k R^k(t) + C  \qquad (t > a^k), \\
  R^k(a^k) &= r_i^k(a^k) = 0, \\
  \dot {R}^k(a^k) &= \dot{r}_i^k(a^k) =: v_i^k.
\end{aligned}
\right.
\]
This differential equation can be solved explicitly. The solution is
\[
R^k(t) = \frac{v_i^k}{\sqrt{k}} \sin(\sqrt{k}(t-a^k)) - \frac{C}{k} \cos(\sqrt{k}(t-a^k)) + \frac{C}{k}.
\]
Let 
\[ 
  t^k 
  := \inf \{ t > a^k~;~ R^k(t) \leq 0 \} 
  = \frac{2}{\sqrt{k}}\left(\arctan\left(\frac{-v_i^k\sqrt{k}}{C}\right) + \pi \right) + a^k.
\]
Note that $w^k(t) := R^k(t) - r_i^k(t)$ satisfies
\[
\left\{
\begin{aligned}
  \frac{d^2}{dt^2} w^k(t) &\geq 0  \qquad (t \in (a^k, b^k)) \\
  w^k(a^k) &= 0, \\
  \dot w^k(a^k) &= 0,
\end{aligned}
\right.
\]
and thus $0 \le w^k(t) = R^k(t) - r_i^k(t) \leq R^k(t)$ for all $t \in (a^k, b^k)$. Hence, $t_k \geq b_k$, and thus 
\[
b^k - a^k \le t^k - a^k = \frac{2}{\sqrt{k}}\left(\arctan\left(\frac{-v_i^k\sqrt{k}}{C}\right) + \pi\right) \leq \frac{\pi}{\sqrt{k}} \to 0
\]
as $k\to\infty$. This proves \eqref{alam3}.

Next, we prove \eqref{alam4}. We set 
\[ e^k(t) := \frac{1}{2} \dot{r}_i^k(t)^2 + \frac k2 r_i^k(t)^2 \] 
as the energy. We recall from \eqref{alam3:pf:1} that $\ddot r_i^k(t) 
= -k r_i^k(t) + g^k(t)$
for all $t \in (a^k, b^k)$, where $\| g^k \|_{C(B(t_0, \delta))}$ is uniformly bounded in $k$. Then,
\[
\begin{aligned}
\dot{e}^k(t) &=  \ddot r_i^k(t) \dot r_i^k(t) + k \dot r_i^k(t) r_i^k(t) \\
&= \big( g^k(t) -k r_i^k(t) \big) \dot r_i^k(t) + k \dot r_i^k(t) r_i^k(t) \\
&= g^k(t) \dot r_i^k(t)
\end{aligned}
\]
for all $t \in (a^k, b^k)$. Hence,
\[
\begin{aligned}
|e^k(b^k) - e^k(a^k)| 
&= \left|\int_{a^k}^{b^k} g^k(t) \dot r_i^k(t) dt \right|\\
& \le \left|b^k - a^k\right| \|g^k\|_{C(B(t_0, \delta))} \|\dot{x}_i^k\|_{\infty},
\end{aligned}
\]
which by \eqref{alam3} and Lemma \ref{l:xk}\eqref{l:xk:xd} vanishes as $k\to\infty$. Since $r_i^k(a^k) = r_i^k(b^k) = 0$, we then also have
\[
|\dot{r}_i^k(b^k)^2 - \dot{r}_i^k(a^k)^2| \to 0
\]
as $k\to\infty$. Since $\dot{r}_i^k(a^k) \to \alpha > 0$ and $\dot{r}_i^k(b^k) \leq 0$ by the definition of $b_k$, \eqref{alam4} follows.

Finally, we prove \eqref{alam5}. Similar to \eqref{alam1:pf:2}, we estimate
\begin{equation} \label{alam5:pf:1}
  |\dot r_i^k(b^k) - \dot r_i(t_0+)|
  \leq |\dot r_i^k(b^k) - \dot r_i^k(t_0 + \delta_4)| + |\dot r_i^k(t_0 + \delta_4) - \dot r_i(t_0 + \delta_4)| + |\dot r_i(t_0 + \delta_4) - \dot r_i(t_0+)|
\end{equation}
for any $\delta_4 > 0$ and any $k$ large enough. By \eqref{alam2} and \eqref{alam4}, $\dot r_i^k(b^k) \leq -\alpha/2 < 0$ for all $k$ large enough. Hence, $r_i^k < 0$ on $(b^k, b^k + \delta^k)$ for some $\delta^k > 0$. By Lemma \ref{l:xk}\eqref{l:xk:xddCI}, there exists $\delta_5 > 0$ such that $\delta^k \geq \delta_5$, and such that $\dot r_i^k$ is Lipschitz continuous on $[b_k, b_k + \delta_5]$ with Lipschitz constant $L$ independent of $k$. Hence, for $\delta_4 \leq \delta_5$ and any $k$ large enough, the first term in \eqref{alam5:pf:1} can be estimated as
\[
  |\dot r_i^k(b^k) - \dot r_i^k(t_0 + \delta_4)|
  \leq L |b^k - (t_0 + \delta_4)|
  \leq L (|b^k - t_0| + \delta_4).
\] 
For the second term in \eqref{alam5:pf:1}, if $\delta_4$ is not in a certain nullset, it follows from $r_k \to r$ pointwise a.e.\ on $B(t_0, \delta)$ that 
\[
|\dot r_i^k(t_0 + \delta_4) - \dot r_i(t_0 + \delta_4)| \to 0
\]
as $k \to \infty$. Collecting our findings above,
\begin{equation*} 
  \lim_{k \to \infty} |\dot r_i^k(b^k) - \dot r_i(t_0+)|
  \leq L \delta_4 + |\dot r_i(t_0 + \delta_4) - \dot r_i(t_0+)|.
\end{equation*}
Since $\delta_4$ can be taken arbitrarily small, we conclude \eqref{alam5}.

Until here, we have proven \eqref{r:equal} under the additional assumption that $\dot r_i(t_0-) > 0$. Analogously, one can prove \eqref{r:equal} under the alternative assumption $\dot r_i(t_0+) < 0$ by following the same steps, but by working backwards in time. Since $r_i \leq 0$ on $B(t_0, \delta)$ and $r_i(t_0) = 0$, we know a priori that $\dot r_i(t_0-) \geq 0 \geq \dot r_i(t_0+)$. Hence, the only case left to check is $\dot r_i(t_0-) = 0 = \dot r_i(t_0+)$, but in this case \eqref{r:equal} is obvious. This completes the proof of \eqref{r:equal}, which completes the proof of Theorem~\ref{exTh}.

%%%%%%%%%%%%%%%%%%%%%%%%%%%%%%%%%%%%%%%%%%%%%%%%%%%%%%%%%%%%%%%

\section{Local uniqueness of solutions}\label{sec4}
\setcounter{equation}{0}
In this section we prove local uniqueness of solutions to Problem~\ref{ProbP} as stated in Theorem~\ref{uniTh}. In addition, we give a counterexample to global-in-time uniqueness. 

\subsection{Proof of Theorem~\ref{uniTh}}

Let $X, Y$ be two solutions to Problem~\ref{ProbP} and let $T_0$ be as in Theorem~\ref{uniTh}. We prove Theorem~\ref{uniTh} by contradiction; suppose that there exists a $T_1\in [0,T_0)$ and an $i \in \{1,\ldots,n\}$ such that for some decreasing sequence $t_\ell \downarrow T_1$ as $\ell \to \infty$ there holds
\begin{equation}\label{uni}
\left\{
\begin{aligned}
X(t) &= Y(t)\qquad &(0\le t\le T_1),\\
x_i(t_\ell) &\neq y_i(t_\ell)\qquad &(\ell=1,2,\ldots).\\
\end{aligned}
\right.
\end{equation} 
We split two cases:
\begin{description}
  \item[Case 1:] $x_i(T_1)\in \Omega$. Let $\delta > 0$ such that the ball $B_{\delta}(x_i(T_1))$ is contained in $\Omega$. Since $x_i$ and $y_i$ are continuous, there exists $\eta>0$ such that $x_i(t), y_i(t)\in B_{\delta}(x_i(T_1))\subset \Omega$ for all $t\in [T_1, T_1+\eta)$, and thus $I^i_{\partial\Omega} \cap [T_1, T_1+\eta) = \emptyset$. Then, since $\supp \rho_i \subset I^i_{\partial\Omega}$, there exists a unique classical solution for Problem~\ref{ProbP} on $(T_1,T_1+\eta)$ from the Lipschitz condition \eqref{2}. This contradicts \eqref{uni}.
  
  \item[Case 2:] $x_i(T_1)\in \partial\Omega$ with $\alpha:= \dot{x}_i(T_1-)\cdot\nu(x_i(T_1))>0$. From \eqref{uni} we note that $\alpha = \dot{y}_i(T_1-)\cdot\nu(y_i(T_1))$. Then, from the reflection rule \eqref{reflection}, we obtain $\dot{x}_i(T_1+)\cdot\nu(x_i(T_1))=\dot{y}_i(T_1+)\cdot\nu(y_i(T_1))=-\alpha <0$. Hence, there exists an $\eta>0$ such that $x_i(t), y_i(t)\in \Omega$ for all $T_1<t<T_1+\eta$. Then, a similar argument as in Case 1 yields a contradiction with \eqref{uni}.
\end{description}

\subsection{Counterexample to global-in-time uniqueness}
\label{ss:ce}

In this section we motivate by means of an example that Definition \ref{Def} cannot provide a unique solution for the general $\Omega$ and $F$ which we consider. Our example is motivated by the one given in Section 3b of \cite{schatzman1978class}. That example is conceptualized by thinking of a tennis player who succeeds in making his ball bounce higher and higher from a rest position by only hitting the ball downwards. The idea of this counterexample goes as least back to \cite{taylor1976grazing}. However, in these references the examples are given in an abstract way; we aim to construct a more explicit version of the example.

Our counterexample is set on the one-dimensional ($m=1$) halfline $\Omega:=\{x\in \R~;~ x>0\}$ with $x^\circ = x_1^\circ = 0$ and $v^\circ = v_1^\circ = 0$ as initial conditions. This initial condition corresponds in the setting of Theorem \ref{uniTh} to $T_0 = 0$. While this initial condition is technically not allowed in Theorem \ref{exTh}, it can easily be obtained by having the particle start inside $\Omega$ at $t = 0$, choose an $F$ so that the particle enters $\partial \Omega$ at $0$ speed at $t = 1/2$, keep the particle on the boundary until $t = 1$, and then shift time backwards by $1$ unit. For any given $L \in \N$, we will construct a function $F(t):=F_1(t,x(t))$ of class $C^L$ which is non-positive (this is in line with the setting of the tennis player). For any such $F$, the $0$-function is a solution according to Definition \ref{Def} (in this case, the density of $\rho = \rho_1$ is $- F(t)$), and thus it remains to construct a non-zero solution.  

Before constructing an $F$ which results in a non-zero solution, we remark that our one-dimensional setup is not restrictive. Indeed, simply by adding a second dimension to our current setup, $\Omega$ becomes the halfplane, and $x_2$ can be considered as constant. Our example will also trivially extend to the case in which $v_2^\circ \neq 0$, in which case $x_2(t)$ changes linearly in time. In this case, $F$ need not explicitly depend on time, and instead dependence on $x_2(t)$ suffices. More generally, our example can be extended to higher dimensions and a curved boundary $\partial \Omega$.

Next we construct $F$ in our one-dimensional setting such that Definition \ref{Def} has a non-zero solution $x(t)$. We do this by tying together rescaled and translated versions of an auxiliary problem where a particle with position $z(t)$ only bounces at the start and end time. Our goal in this auxiliary problem is that the speed of impact $-v_1$ at the end time $1$ is larger than the initial speed $v_0$ at $t=0$, i.e., $v_1-v_0 > 0$. 

To reach this goal, let $f\in C_0^{\infty}(0,1)$ with $f(t)\ge 0$ satisfy
\begin{equation} \label{f:int:cond}
  \int_0^1 \left( 2s - 1 \right) f(s)ds 
  >0.
\end{equation}
For $v_0 > 0$ to be chosen later, let $z$ be the solution to
\[
\left\{
\begin{aligned}
z''(t)&=-f(t),\\
z(0)&=0,\\
z'(0)&=v_0.
\end{aligned}
\right.
\]
It is easy to see that $z(t) = v_0 t - \int_0^t(t-s)f(s)ds$. Next we choose $v_0$ such that the particle hits the boundary again at $t = 1$, i.e., such that $z(1) = 0$. This yields
\begin{equation*}
  v_0=\int_0^1(1-s)f(s)ds 
\end{equation*}
We then find that
\begin{equation*} 
  v_1 = -z'(1) = \int_0^1sf(s)ds.
\end{equation*}
Finally, to test the requirement $v_1-v_0 > 0$, we compute
\[
  v_1-v_0
  = \int_0^1 \left( 2s - 1 \right) f(s)ds,
\]
and observe from \ref{f:int:cond} that this value is indeed positive.

Next we rescale the auxiliary problem. For $a,b > 0$ to be chosen later, let
\[
y(t) := b z\left(\frac{t}{a}\right)\qquad (t\in [0,a]).
\]
Then, $y$ satisfies
\[
\left\{
\begin{aligned}
y''(t) &=-\frac{b}{a^2} f\left(\frac{t}{a}\right),\\
y(0)&=y(a)=0,\\
y'(0)&=\frac{b}{a}v_0,\\
y'(a)&=-\frac{b}{a}v_1.
\end{aligned}
\right.
\]

Next we tie the rescaled problems together by shifting them in space. Figure \ref{nonunique} illustrates the corresponding construction. The reflection rule dictates that minus the velocity at the end time of the scaled auxiliary problem to the left has to equal the velocity at initial time of the scaled auxiliary problem to the right, i.e., $v_0=\frac{b}{a}v_1$. This requires $b < a < 1$. The shifted time interval of the $n$-th bounce to the left is given by $I_n:=(\frac{a^{n+1}}{1-a}, \frac{a^{n}}{1-a})$. The resulting function $F$ is
\begin{equation}\label{sF}
F(t)=\left\{
\begin{array}{ll}
-\left(\dfrac{b}{a^2}\right)^n f\left(\left(t-\dfrac{a^{n+1}}{1-a}\right) a^{-n}\right) &\qquad (t\in I_n, n\in\mathbb{Z}), \\
0 &\qquad \text{otherwise},
\end{array}
\right.
\end{equation}
where we allow for $t < 0$. For this non-positive $F$, a non-zero solution $x$ to Problem \eqref{ProbP} is
\begin{equation}\label{ce:sol}
x(t)=\left\{
\begin{array}{ll}
b^n z\left(\left(t- \dfrac{a^{n+1}}{1-a}\right)a^{-n}\right) &\qquad (t\in I_n, n\in\mathbb{Z}),\\
0 &\qquad \text{otherwise},
\end{array}
\right.
\end{equation}
which we also extend by $0$ for $t < 0$.

\begin{figure}[htbp]
  \begin{center} 
    \includegraphics[width=140mm]{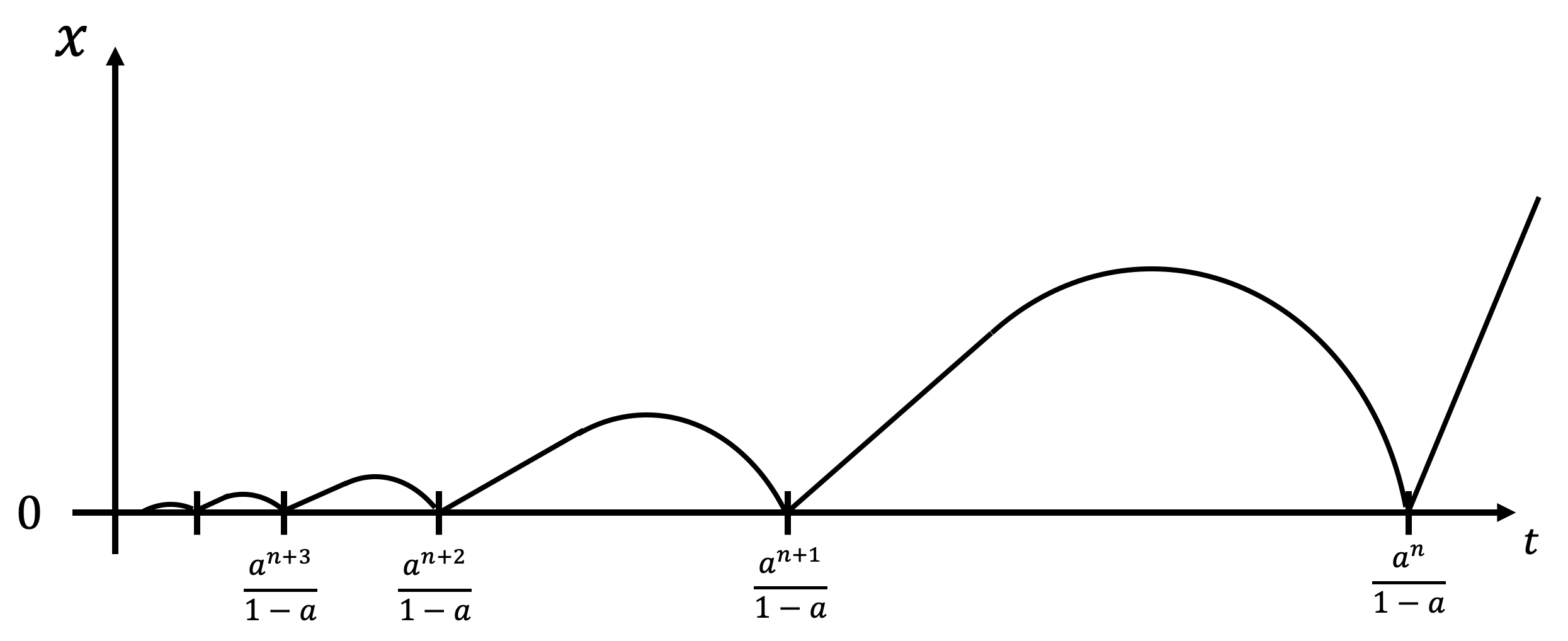}
  \end{center}
  \caption{Sketch of the non-zero solution $x(t)$.}
  \label{nonunique}
\end{figure}

Next we check that $x$ is indeed a solution to Problem \ref{ProbP}. By construction, $x$ is a classical solution on each $I_n$, and at the endpoints of each $I_n$ the reflection rule is satisfied. Therefore, Definition \ref{Def}(ii),(iii) is satisfied. Checking Definition \ref{Def}(i) requires further care. We verify it on the time interval $[-1,T]$, where $T = \frac12 \frac{1+a}{1-a}$, which is the midpoint of the interval $I_0$. We observe that $x$ is continuous and piecewise smooth on $[-1,0)$ and on each $I_n$. We also observe that
\begin{equation}\label{x:CvB:bds}
  \| x \|_{C(\overline{I_n})} \leq \| z \|_{C([0,1])} b^n, \quad 
  \| x' \|_{L^\infty(I_n)} \leq \| z' \|_{C([0,1])} \left(\frac{b}{a}\right)^n, \quad 
  \| x'' \|_{L^\infty(I_n)} \leq \| z'' \|_{C([0,1])} \left(\frac{b}{a^2}\right)^n.
\end{equation}
Hence, $x$ is Lipschitz continuous on $[-1,T]$, and thus $x \in W^{1,\infty}(-1,T)$. To check that $\dot x \in BV(-1,T)$, we construct a right-continuous $u \in BV(-1,T)$, and show that it is a representative for $\dot x$. With this aim, we set $t_n := \tfrac{a^{n+1}}{1-a}$ as the left endpoint of $I_n$, and define the Borel measure
\[
  \mu (A) 
  := \sum_{n=0}^\infty \bigg( 2 v_0 \left(\frac{b}{a}\right)^n \delta_{t_n} (A) + \int_{A \cap I_n} \big( x |_{I_n} \big)''(t) dt \bigg)
  \qquad \text{for all intervals } A \subset (-1,T).
\]
Since 
\begin{align*}
  |\mu|((-1,T))
  &\leq \sum_{n=0}^\infty \bigg( 2 v_0 \left(\frac{b}{a}\right)^n + \int_{I_n} \| x'' \|_{L^\infty(I_n)} dt \bigg) \\
  &\leq \sum_{n=0}^\infty \bigg( 2 v_0 \left(\frac{b}{a}\right)^n + \| z'' \|_{C([0,1])}  \left(\frac{b}{a^2}\right)^n |I_n| \bigg) \\
  &\leq C \sum_{n=0}^\infty \left(\frac{b}{a}\right)^n,
\end{align*}
it holds that $\mu \in \mathcal M_+((-1,T))$ if $b < a$. Hence, the right-continuous function
\[
  u(t) := x'(T) - \mu((t,T))
\]
is in $BV(-1,T)$. It is easy to check (for instance, by induction over $I_n$) that $u(t) = \dot x(t)$ for all $t > 0$. Since $u$ is right-continuous, we obtain from \eqref{x:CvB:bds} that
\begin{equation}\label{limitat0}
  u(0) = u(0+) = \dot x(0+) = 0.
\end{equation}
Finally, since $|\mu|((-1,0]) = 0$, $u$ is constant on $(-1,0]$. Putting these findings on $u$ together, we conclude that $\dot x = u \in BV(-1,T)$. This concludes the proof that $x$ satisfies Definition \ref{Def}.

Next we check the regularity of $F$. Since $f$ is smooth and supported inside $(0,1)$, its scaled and shifted copies are also smooth and supported inside $I_n$. Therefore, $F$ is smooth on $(-\infty, 0) \cup (0,T)$. The regularity of $F$ at $t=0$ requires further care, and can be checked similar to that of $x$. Indeed,
\[
  \| F^{(\ell)} \|_{C(\overline{I_n})} \leq \| f^{(\ell)} \|_{C([0,1])} \left(\frac{b}{a^{\ell + 2}}\right)^n.
\]
Thus, if $b < a^L$, then $F$ is $L$ times continuously differentiable at $0$, and thus $F \in C^L((-\infty, T])$.

Next we complete our counterexample by choosing $a,b$. All conditions on $a,b$ which we have used can be summarized by the two conditions $0 < b < a^L < 1$ and $v_0 = \frac ba v_1$. Since such $a,b$ are easy to construct for $L =1$, we focus on the case $L \geq 2$. By choosing 
\[ a \geq \left( \frac{v_0}{v_1} \right)^{\tfrac1{L-1}} \] 
and $b := a v_0 / v_1$, it is easy to check that both conditions on $a,b$ are satisfied.

%%%%%%%%%%%%%%%%%%%%%%%%%%%%%%%%%%%%%%%%%%%%%%%%%%%%%%%%%%%%%%%%%%%%%%%%%%%%%%%%%%%%%%%%%%%%%%%%%%%%%%%%%%%%%%%%%%%%%%%%%%%%%%
\section{Energy conservation}\label{sec5}
\setcounter{equation}{0}

In the introduction we mentioned that the solution concept to Problem \ref{ProbP} is different from that in \cite{schatzman1978class}. The solution concept in \cite{schatzman1978class} is obtained from Definition \ref{Def} by replacing condition (iii) by two conditions: first, the speed of each particle has to be continuous, and second, the energy has to be conserved. Proposition \ref{speed} states that Definition \ref{Def} yields continuous speed for each particle. In this section we investigate the second condition about the conservation of energy.

While the setting in \cite{schatzman1978class} is restricted to a conservative force, we can still extend this concept to nonconservative force fields $F$. We say that a solution $X$ of Problem \ref{ProbP} by Definition \ref{Def} conserves the energy if for any $i\in \{1,\ldots,n\}$ and all $0\le s_1 < s_2 \le T$, 
\begin{equation}\label{energy}
\frac{1}{2} |\dot{x}_i(s_2)|^2 - \frac{1}{2}|\dot{x}_i(s_1)|=\int_{s_1}^{s_2} F_i(t,X(t))\cdot \dot{x}_i(t) dt.
 \end{equation}
Here, we interpret the term in the right-hand side as the work done by $F_i$ on the time interval $(s_1,s_2)$.

\subsection{Energy conservation in a particular case}

In this section we prove that the energy is conserved if the number of nonzero-velocity collisions with the boundary is finite. This condition allows for particles to move along the boundary, but excludes cases such as the counterexample in Section \ref{ss:ce}. The following theorem makes this statement precise.
 
\begin{Th}\label{conser}
Let $X(t)=\{x_i(t)\}_{i=1}^n$ be a solution to Problem~\ref{ProbP}, and let $\rho_i\in\mathcal{M}_+((0,T))$ for $i=1,\ldots,n$ be the corresponding measures as in Definition~\ref{Def}(ii). If for any $i=1,\ldots,n$ there exist $N \in \N$, $a \in L^1(0,T)$, $t_k \in (0,T)$, $a_k > 0$ for $k = 1,\ldots,N$ such that
\begin{equation}\label{rho}
\rho_i(A) = \sum_{t_{k}\in A} a_{k} + \int_A a(t)dt
\end{equation}
for all open intervals $A \subset (0,T)$, then $X$ conserves the energy, i.e., $X$ satisfies \eqref{energy}.
\end{Th}

We note for a particle $x_i$ that $t_k$ are the times at which it hits the boundary with normal velocity $a_k/2$, and that nonzero parts of $a(t) \geq 0$ correspond to movement along the boundary.

\begin{proof}[Proof of Theorem \ref{conser}]
Take $i\in\{1,\ldots,n\}$ and $0 \leq s_1 < s_2 \leq T$ arbitrarily. By relabelling if necessary we may assume that $0 =: t_0 < t_1 < t_2 < \ldots < t_N < t_{N+1} := T$. We start with the case in which $(s_1,s_2)\subset (t_k,t_{k+1})$ for some $k\in\{1,\ldots,N\}$. First, since the measure $\rho_i$ is absolutely continuous on $(s_1,s_2)$, we obtain from \eqref{weakform} that $\ddot{x}_i = F_i - \chi_{I_{\partial\Omega}^i} (\nu\cdot x_i) \rho_i \in L^1(s_1,s_2)$, and thus $\dot{x}_i\in W^{1,1}(s_1,s_2;\R^m)$.
Second, we define the cut-off function $\eta_\e$ with $\supp \eta_\e \subset [s_1,s_2]$ by
\[
\eta_\e(t):=\left\{
\begin{aligned}
&(t-s_1)/\e& (s_1\le t < s_1 +\e),\\
&1 & (s_1 +\e\le t < s_2-\e),\\
&(s_2-t)/\e & (s_2-\e\le t \le s_2),
\end{aligned} 
\right.
\]
which is in $W_0^{1,\infty}(s_1,s_2)$. Then, $\psi_\e:=\eta_\e \dot{x}_i \in W_0^{1,1}(0,T,\R^m)$. Hence, by Remark~\ref{remarkop} and the weak form \eqref{weakform} it follows that
\begin{equation}\label{pf:csv:1}
    -\int_0^T \dot{x}_i(t)\cdot \dot{\psi_\e}(t) dt = \int_0^T F_i(t,X(t)) \cdot \psi_\e(t) dt - \int_{I_{\partial\Omega}^i}\nu(x_i(t))\cdot \psi_\e(t)d\rho_i(t).
\end{equation}
In preparation for passing to $\e\to 0$, we rewrite the left-hand side as
\[
\begin{aligned}
-\int_0^T \dot{x}_i(t)\cdot \dot{\psi_\e}(t) dt &= -\int_0^T  \dot{x}_i(t)\cdot (\dot{\eta}_\e(t) \dot{x}_i(t)  + \eta_\e(t) \ddot{x}_i(t)) dt\\
&= -\int_{s_1}^{s_2} \eta_\e(t) \frac{d}{dt}\left(\frac{1}{2}|\dot{x}_i(t)|^2\right) dt - \int_{s_1}^{s_2} \dot{\eta}_\e(t) |\dot{x}_i(t)|^2 dt\\
&= \int_{s_1}^{s_2} \dot{\eta}_\e(t) \frac{1}{2} |\dot{x}_i(t)|^2 dt -\int_{s_1}^{s_2} \dot{\eta}_\e(t) |\dot{x}_i(t)|^2 dt\\
&= \frac{1}{\e} \int_{s_2-\e}^{s_2}  \frac{1}{2}|\dot{x}_i(t)|^2 dt - \frac{1}{\e}\int_{s_1}^{s_1 + \e} \frac{1}{2}|\dot{x}_i(t)|^2 dt.
\end{aligned}
\]
For the right-hand side of \eqref{pf:csv:1}, we use $\supp \eta_\e = [s_1,s_2]$, $\eta_\e(s_1) = \eta_\e(s_2) = 0$ and $d\rho_i(t)=a(t)dt$ on $(s_1,s_2)$ to rewrite it as
\[
\begin{aligned}
\int_{s_1}^{s_2} \eta_\e(t) \, F_i(t,X(t))\cdot \dot{x}_i(t) dt - \int_{s_1}^{s_2} \eta_\e(t) \, \chi_{I_{\partial\Omega}^i}(t) \nu(x_i(t))\cdot \dot{x}_i(t) \, a(t)dt
\end{aligned}
\]
where
\[
\chi_{I_{\partial\Omega}^i}(t)=\left\{
\begin{aligned}
&1 &(t\in I_{\partial\Omega}^i),\\
&0 &(t\notin I_{\partial\Omega}^i).
\end{aligned}
\right.
\]
We claim that for any $\e > 0$
\begin{equation} \label{pf:csv:2}
  \int_{s_1}^{s_2} \eta_\e(t) \, \chi_{I_{\partial\Omega}^i}(t) \nu(x_i(t))\cdot \dot{x}_i(t) \, a(t)dt = 0.
\end{equation}
Indeed, if $(s_1,s_2)\cap I_{\partial\Omega}^i = \emptyset$, then $\chi_{I_{\partial\Omega}^i}(t)=0$ for all $t\in (s_1, s_2)$, and thus \eqref{pf:csv:2} holds. If $(s_1,s_2)\cap I_{\partial\Omega}^i \neq \emptyset$, then take any $s\in (s_1,s_2)\cap I_{\partial\Omega}^i$. If $\alpha := \dot{x}_i(s- ) \cdot \nu(x_i(s)) >0$, then the reflection rule \eqref{reflection} implies that $\rho_i(\{s\}) = 2 \alpha$, which contradicts with $\rho_i$ being absolutely continuous on $(s_1,s_2)$. Therefore, $\dot{x}_i(s - ) \cdot \nu(x_i(s)) =0$. Then, by \eqref{reflection} we have $\dot{x}_i(s+) = \dot{x}_i(s-)$. Thus, $\nu(x_i(s))\cdot\dot{x}_i(s)=0$ for all $s\in (s_1,s_2)\cap I_{\partial\Omega}^i$, which implies \eqref{pf:csv:2}. 

Collecting our results above, the passage to the limit $\e \to 0$ in \eqref{pf:csv:1} yields \eqref{energy}. This concludes the prove in the case where $\rho_i$ is absolutely continuous on $(s_1,s_2)$.

Next, we consider the general case $0 \leq s_1 < s_2 \leq T$. Let $k$ and $\ell$ be such that $s_1 \in [t_{k-1}, t_k)$ and $s_2 \in (t_\ell, t_{\ell+1}]$. Then, we apply the result above to obtain the energy conservation \eqref{energy} on the intervals
\[
(s_1, t_k), \ (t_k, t_{k+1}), \ldots, (t_{\ell-1}, t_\ell), (t_\ell, s_2).
\]
Adding up the energy conservation on each of these intervals and using  Proposition~\ref{speed}, we obtain that \eqref{energy} holds on $(s_1,s_2)$.
\end{proof}

\subsection{Energy conservation in the counterexample in Section~\ref{sec4}}

Since Theorem \ref{conser} shows that solutions $X$ to Problem \ref{ProbP} conserve the energy in particular cases, it is interesting to investigate whether adding an energy conservation condition to Definition \ref{Def} could resolve the issue of non-uniqueness of solutions. In particular, the nontrivial solution constructed in Section \ref{ss:ce} does not satisfy the conditions of Theorem \ref{conser}. 

We show in this section that the answer is negative; adding the energy conservation condition \eqref{energy} to Definition \ref{Def} still allows for multiple solutions. We show this by computing \eqref{energy} for the two solutions constructed in Section \ref{ss:ce}. In the remainder of this section, we use the notation from Section \ref{ss:ce}.

The trivial solution obviously satisfies \eqref{energy}. Hence, we focus on the nontrivial solution $x(t)$ defined in \eqref{ce:sol}. For $(s_1,s_2)\subset (-1,0)$, since $x |_{(s_1,s_2)} \equiv 0$, \eqref{energy} holds. For $(s_1,s_2)\subset (0,T)$ with $s_1 > 0$, $x$ hits $\partial \Omega$ only finitely many times on $(s_1,s_2)$, and thus \eqref{energy} is guaranteed to hold by Theorem \ref{conser}.
Since $|\dot{x}|$ is continuous on $[-1,T]$ (see Proposition \ref{speed}) and $|\dot{x}|(0-) = 0$, we have that $|\dot{x}|(0+) = 0$, and thus passing to the limit $s_1 \to 0+$ we obtain that \eqref{energy} holds on $(0,s_2)$.

Finally, for the remaining case in which $0\in (s_1,s_2)$, we can combine the results above on $(s_1,0)$ and $(0,s_2)$ in a similar manner as in the last paragraph of the proof of Theorem \ref{conser}. Hence, for any $(s_1,s_2)\subset (-1,T)$, \eqref{energy} holds. Thus, both solutions constructed in Section \ref{ss:ce} satisfy the energy conservation.

\section{Conclusion}
\label{sec:conc}

In Definition \ref{Def} we have given a proper meaning to the particle dynamics described by \eqref{P:informal} with elastic collisions at $\partial \Omega$. We proved that this definition leads to global-in-time existence of solutions; see Theorem \ref{exTh}. We also showed that a solution is unique up to the first time at which a particle hits $\partial \Omega$ in tangential direction; see Theorem \ref{uniTh}. At such an event, we showed by means of an example that the solution can be continued in two different manners, and that both such solutions conserve the energy. We leave it to future work to show that any solution conserves the energy, and to find a modification to Definition \ref{Def} which yields uniqueness of solutions.

\section*{Acknowledgements}

MK and PvM gratefully acknowledge support from JSPS KAKENHI Grant Number 20KK0058.
PvM gratefully acknowledges support from JSPS KAKENHI Grant Number 20K14358.


\begin{thebibliography}{KvMY19}

\bibitem[AASS12]{asai2012stabilized}
M.~Asai, A.M. Aly, Y.~Sonoda, and Y.~Sakai.
\newblock A stabilized incompressible {SPH} method by relaxing the density
  invariance condition.
\newblock {\em Journal of Applied Mathematics}, 2012, 2012.

\bibitem[AFP00]{ambrosio2000functions}
L.~Ambrosio, N.~Fusco, and D.~Pallara.
\newblock {\em Functions of bounded variation and free discontinuity problems},
  volume 254.
\newblock Clarendon Press Oxford, 2000.

\bibitem[BP83]{buttazzo1983approximation}
G.~Buttazzo and D.~Percivale.
\newblock On the approximation of the elastic bounce problem on riemannian
  manifolds.
\newblock {\em Journal of Differential Equations}, 47(2):227--245, 1983.

\bibitem[CP80]{carriero1980uniqueness}
M.~Carriero and E.~Pascali.
\newblock Uniqueness of the one-dimensional bounce problem as a generic
  property in {$L^1 ([0, T]; \mathbb{R})$}.
\newblock {\em Quaderni di Matematica}, 1980(12), 1980.

\bibitem[EG15]{evans2015measure}
L.~C. Evans and R.~F. Gariepy.
\newblock {\em Measure theory and fine properties of functions}.
\newblock CRC press, 2015.

\bibitem[GM77]{gingold1977smoothed}
R.~A. Gingold and J.~J. Monaghan.
\newblock Smoothed particle hydrodynamics: theory and application to
  non-spherical stars.
\newblock {\em Monthly Notices of the Royal Astronomical Society},
  181(3):375--389, 1977.

\bibitem[GT85]{gilbarg2015elliptic}
D.~Gilbarg and N.~S. Trudinger.
\newblock {\em Elliptic Partial Differential Equations of Second Order, 2nd
  Edition}.
\newblock Springer-Verlag, New York, 1985.

\bibitem[Har82]{hartmanordinary}
P.~Hartman.
\newblock {\em Ordinary Differential Equations}.
\newblock Boston: Birkh{\"a}user, 1982.

\bibitem[Kim08]{kimura2008geometry}
M.~Kimura.
\newblock Geometry of hypersurfaces and moving hypersurfaces in
  {$\mathbb{R}^m$} for the study of moving boundary problems.
\newblock {\em Topics in Mathematical Modeling}, 4:39--93, 2008.

\bibitem[KvMY19]{kimura2019particle}
M.~Kimura, P.~van Meurs, and Z.~Yang.
\newblock Particle dynamics subject to impenetrable boundaries: existence and
  uniqueness of mild solutions.
\newblock {\em SIAM Journal on Mathematical Analysis}, 51(6):5049--5076, 2019.

\bibitem[LL04]{li2007meshfree}
S.~F. Li and W.~K. Liu.
\newblock {\em Meshfree Particle Methods}.
\newblock Springer-Verlag Berlin Heidelberg, 2004.

\bibitem[LL10]{liu2010smoothed}
M.B. Liu and G.R. Liu.
\newblock Smoothed particle hydrodynamics ({SPH}): an overview and recent
  developments.
\newblock {\em Archives of Computational Methods in Engineering}, 17(1):25--76,
  2010.

\bibitem[Luc77]{lucy1977numerical}
L.B. Lucy.
\newblock A numerical approach to the testing of the fission hypothesis.
\newblock {\em The Astronomical Journal}, 82:1013--1024, 1977.

\bibitem[Mon92]{monaghan1992smoothed}
J.J. Monaghan.
\newblock Smoothed particle hydrodynamics.
\newblock {\em Annual Review of Astronomy and Astrophysics}, 30(1):543--574,
  1992.

\bibitem[Per91]{percivale1991uniqueness}
D.~Percivale.
\newblock Uniqueness in the elastic bounce problem, {II}.
\newblock {\em Journal of Differential Equations}, 90(2):304--315, 1991.

\bibitem[Sch78]{schatzman1978class}
M.~Schatzman.
\newblock A class of nonlinear differential equations of second order in time.
\newblock {\em Nonlinear Analysis: Theory, Methods and Applications},
  2(3):355--373, 1978.

\bibitem[Sch98]{schatzman1998uniqueness}
M.~Schatzman.
\newblock Uniqueness and continuous dependence on data for one-dimensional
  impact problems.
\newblock {\em Mathematical and Computer Modelling}, 28(4-8):1--18, 1998.

\bibitem[Tay76]{taylor1976grazing}
M.~E. Taylor.
\newblock Grazing rays and reflection of singularities of solutions to wave
  equations.
\newblock {\em Communications on Pure and Applied Mathematics}, 29:1--38, 1976.

\end{thebibliography}
\end{document}